\documentclass[oneside,english]{amsart}
\usepackage[T1]{fontenc}
\usepackage[latin9]{inputenc}
\usepackage{amsthm}
\usepackage{amstext}
\usepackage{amssymb}
\usepackage{verbatim}

\makeatletter
\numberwithin{equation}{section}
\numberwithin{figure}{section}
\theoremstyle{plain}
\newtheorem{thm}{\protect\theoremname}[section]
  \theoremstyle{definition}
  
  \theoremstyle{plain}
  \newtheorem{lem}[thm]{\protect\lemmaname}
  \theoremstyle{definition}
  
  \theoremstyle{definition}
  \newtheorem{remark}[thm]{\protect\remarkname}
  \theoremstyle{plain}
  \newtheorem{prop}[thm]{\protect\propositionname}
   \theoremstyle{plain}
   
   \theoremstyle{plain}
  \newtheorem{cor}[thm]{\protect\corollaryname}
  \theoremstyle{plain}
  \newtheorem*{thm*}{\protect\theoremname}
  \theoremstyle{plain}
  \newtheorem{claim}[thm]{\protect\claimname}
   \theoremstyle{plain}
  \newtheorem*{prop*}{\protect\propositionname}

  \newcounter{casectr}
  \newenvironment{caseenv}
  {\begin{list}{{\itshape\ \protect\casename} \arabic{casectr}.}{%
   \setlength{\leftmargin}{\labelwidth}
   \addtolength{\leftmargin}{\parskip}
   \setlength{\itemindent}{\listparindent}
   \setlength{\itemsep}{\medskipamount}
   \setlength{\topsep}{\itemsep}}
   \setcounter{casectr}{0}
   \usecounter{casectr}}
  {\end{list}}

\makeatother

\usepackage{babel}
  \providecommand{\claimname}{Claim}
  \providecommand{\definitionname}{Definition}
  \providecommand{\examplename}{Example}
  \providecommand{\lemmaname}{Lemma}
  \providecommand{\propositionname}{Proposition}
   \providecommand{\remarkname}{Remark}
  \providecommand{\corollaryname}{Corollary}
  \providecommand{\theoremname}{Theorem}
\providecommand{\theoremname}{Theorem}
\providecommand{\conjecturename}{Conjecture}

\begin{document}
\global\long\def\Inf{\mathrm{Inf}}
\global\long\def\Sym{\mathrm{Sym}}
\global\long\def\Per{\mathrm{Per}}
\global\long\def\f{\mathcal{F}}
\global\long\def\l{\mathcal{L}}
\global\long\def\pn{\mathcal{P}\left(\left[n\right]\right)}
\global\long\def\g{\mathcal{G}}
\global\long\def\s{\mathcal{S}}
\global\long\def\m{\mathcal{M}}
\global\long\def\mh{\mu_{\frac{1}{2}}}
\global\long\def\mp{\mu_{p}}
\global\long\def\j{\mathcal{J}}
\global\long\def\d{\mathcal{D}}
\global\long\def\p{\mathcal{P}}
\global\long\def\mpo{\mu_{p_{0}}}
\global\long\def\fpp{f_{p}^{p_{0}}}
\global\long\def\llll{\l_{\mu}}
\global\long\def\h{\mathcal{H}}
\global\long\def\n{\mathbb{N}}
\global\long\def\a{\mathcal{A}}
\global\long\def\b{\mathcal{B}}
\global\long\def\sf{f_{2}}
\global\long\def\bin{\mathrm{Bin}}
\global\long\def\C{C_{2}}

\title{On a Biased Edge Isoperimetric Inequality for the Discrete Cube}

\author{David Ellis}
\address{David Ellis, School of Mathematical Sciences, Queen Mary, University of London, Mile End Road, London E1 4NS, UK.}
\email{d.ellis@qmul.ac.uk}
\author{Nathan Keller}
\address{Nathan Keller, Department of Mathematics, Bar Ilan University, Ramat Gan 5290002, Israel.}
\email{nathan.keller27@gmail.com}
\author{Noam Lifshitz}
\address{Noam Lifshitz, Department of Mathematics, Bar Ilan University, Ramat Gan 5290002, Israel.}
\email{noamlifshitz@gmail.com}
\thanks{The research of N.K. was supported by the Israel Science Foundation (grant no.
402/13), the Binational US-Israel Science Foundation (grant no. 2014290), and by the Alon Fellowship.}
\date{3rd February 2017}

\begin{abstract}
The `full' edge isoperimetric inequality for the discrete cube $\{0,1\}^n$ (due to Harper, Lindsey, Berstein and Hart) specifies the minimum size of the edge boundary $\partial A$ of a set $A \subset \{0,1\}^n$, as function of $|A|$. A weaker (but more widely-used) lower bound is $|\partial A| \geq |A| \log(2^n/|A|)$, where equality holds whenever $A$ is a subcube. In 2011, the first author obtained a sharp `stability' version of the latter result, proving that if $|\partial A| \leq |A| (\log(2^n/|A|)+\epsilon)$, then there exists a subcube $C$ such that $|A \Delta C|/|A| = O(\epsilon /\log(1/\epsilon))$.

The `weak' version of the edge isoperimetric inequality has the following well-known generalization for the `$p$-biased' measure $\mu_p$ on the discrete cube: if $p \leq 1/2$, or if $0 < p < 1$ and $A$ is monotone increasing, then $p\mu_p(\partial A) \geq \mu_p(A) \log_p(\mu_p(A))$.

In this paper, we prove a sharp stability version of the latter result, which generalizes the aforementioned result of the first author. Namely, we prove that if $p\mu_p(\partial A) \leq \mu_p(A) (\log_p(\mu_p(A))+\epsilon)$, then there exists a subcube $C$ such that $\mu_p(A \Delta C)/\mu_p(A) = O(\epsilon' /\log(1/\epsilon'))$, where $\epsilon':=\epsilon \ln (1/p)$. This result is a central component in recent work of the authors proving sharp stability versions of a number of Erd\H{o}s-Ko-Rado type theorems in extremal combinatorics, including the seminal `complete intersection theorem' of Ahlswede and Khachatrian.

In addition, we prove a biased-measure analogue of the `full' edge isoperimetric inequality, for monotone increasing sets, and we observe that such an analogue does not hold for arbitrary sets, hence answering a question of Kalai. We use this result to give a new proof of the `full' edge isoperimetric inequality, one relying on the Kruskal-Katona theorem.
\end{abstract}

\maketitle

\section{Introduction}

Isoperimetric inequalities are of ancient interest in mathematics. In general, an isoperimetric inequality gives a lower bound on the `boundary-size'
of a set of a given `size', where the exact meaning of these words varies according to the problem. In the last fifty years, there has been a great deal of interest in {\em discrete} isoperimetric inequalities. These deal with the `boundary' of a set $A$ of vertices in a graph $G=(V,E)$ -- either the {\em edge boundary} $\partial A$, which consists of the set of edges of $G$ that join a vertex in $A$ to a vertex in $V \setminus A$, or the {\em vertex boundary} $b(A)$, which consists of the set of vertices of $V \setminus A$ that are adjacent to a vertex in $A$.

\subsection{The edge isoperimetric inequality for the discrete cube, and some stability versions thereof}

A specific discrete isoperimetric problem which attracted much interest due to its numerous applications is the edge isoperimetric problem for the $n$-dimensional discrete cube, $Q_n$. This is the graph with vertex-set $\{0,1\}^n$, where two 0-1 vectors are adjacent if they differ in exactly one coordinate. The edge isoperimetric problem for $Q_n$ was solved by Harper \cite{Harper}, Lindsey \cite{Lindsey}, Bernstein \cite{Bernstein}, and Hart \cite{Hart}. Let us describe the solution. We may identify $\{0,1\}^n$ with the power-set $\pn$ of $[n]: = \{1,2,\ldots,n\}$, by identifying a 0-1 vector $(x_1,\ldots,x_n)$ with the set $\{i \in [n]:\ x_i=1\}$. We can then view $Q_n$ as the graph with vertex set $\pn$, where two sets $S, T \subset [n]$ are adjacent if $|S \Delta T|=1$. The {\em lexicographic ordering} on $\pn$ is defined by $S > T$ iff $\min(S \Delta T) \in S$. If $m \in [2^n]$, the {\em initial segment of the lexicographic ordering on $\pn$ of size $m$} (or, in short, the {\em lexicographic family of size $m$}) is simply the $m$ largest elements of $\pn$ with respect to the lexicographic ordering.
Harper, Bernstein, Lindsey and Hart proved the following.
\begin{thm}[The `full' edge isoperimetric inequality for $Q_n$]
\label{thm:edge-iso}
If $\f \subset \p([n])$ then $|\partial \f| \geq |\partial \l|$, where $\l \subset \pn$ is the initial segment of the lexicographic ordering of size $|\f|$.
\end{thm}
A weaker, but more convenient (and, as a result, more widely-used) lower bound, is the following:
\begin{cor}[The weak edge isoperimetric inequality for $Q_n$]
\label{cor:edge-iso}
If $\f \subset \p([n])$ then
\begin{equation}\label{Eq:weak-iso}
|\partial \f| \geq |\f|\log_2(2^n/|\f|).
\end{equation}
\end{cor}
Equality holds in (\ref{Eq:weak-iso}) iff $\f$ is a subcube, so (\ref{Eq:weak-iso}) is sharp only when $|\f|$ is a power of 2.

When an isoperimetric inequality is sharp, and the extremal sets are known, it is natural to ask whether the inequality is also `stable' --- i.e., if a set has boundary of size `close' to the minimum, must that set be `close in structure' to an extremal set?

For Corollary~\ref{cor:edge-iso}, this problem was studied in several works. Using a Fourier-analytic argument, Friedgut, Kalai and Naor~\cite{FKN} obtained a stability result for sets of size $2^{n-1}$, showing that if $\f \subset \pn$ with $|\f| = 2^{n-1}$ satisfies $|\partial \f| \leq (1+\epsilon)2^{n-1}$, then $|\f \Delta \mathcal{C}|/2^n= O(\epsilon)$ for some codimension-1 subcube $\mathcal{C}$. (The dependence upon $\epsilon$ here is almost sharp, viz., sharp up to a factor of $\Theta(\log(1/\epsilon))$). Bollob\'as, Leader and Riordan (unpublished) proved an analogous result for $|\f| \in \{2^{n-2},2^{n-3}\}$, also using a Fourier-analytic argument. Samorodnitsky \cite{Samorodnitsky09} used a result of Keevash~\cite{Keevash08} on the structure of $r$-uniform hypergraphs with small shadows, to prove a stability result for all $\f \subset \pn$ with $\log_2|\f| \in \mathbb{N}$ (i.e., all sizes for which Corollary~\ref{cor:edge-iso} is tight), under the rather strong condition $|\partial \f| \leq (1+O(1/n^4))|\partial \l|$. In \cite{Ellis}, the first author proved the following stability result (which implies the above results), using a recursive approach and an inequality of Talagrand \cite{Talagrand} (which was proved via Fourier analysis).
\begin{thm}[\cite{Ellis}]
\label{thm:e}
There exists an absolute constant $c>0$ such that the following holds. Let $0 \leq \delta < c$. If $\f \subset \pn$ with $|\f| = 2^{d}$ for some \(d \in \mathbb{N}\), and $|\f \Delta \mathcal{C}| \geq \delta 2^d$ for all $d$-dimensional subcubes $\mathcal{C} \subset \pn$, then
$$|\partial \f| \geq |\partial \mathcal{C}| +2^d \delta \log_{2}(1/\delta).$$
\end{thm}
As observed in \cite{Ellis}, this result is best-possible (except for the condition $0 \leq \delta < c$, which was conjectured to be unnecessary in \cite{Ellis}).

In \cite{LOL}, we obtain the following stability version of Theorem \ref{thm:edge-iso}, which applies to families of arbitrary size (not just a power of 2), and which is sharp up to an absolute constant factor.
\begin{thm}
\label{thm:full-stability}
There exists an absolute constant $C>0$ such that
the following holds. If $\f\subset \pn$ and $\l \subset \p([n])$ is the initial segment of the lexicographic ordering of size $|\f|$, then there exists an automorphism $\sigma$ of $Q_n$ such that
$$|\f\, \Delta\, \sigma(\l)| \leq C(|\partial \f| - |\partial \l|).$$
\end{thm}
The proof uses only combinatorial tools, but is much more involved than the proof of Theorem \ref{thm:e} in \cite{Ellis}.

\subsection{Influences of Boolean functions}

An alternative viewpoint on the edge isoperimetric inequality, which we will use throughout the paper, is via \emph{influences} of Boolean functions. For a function $f:\{0,1\}^n \rightarrow \{0,1\}$, the influence of the $i$th coordinate on $f$ is defined by
\[
 I_i[f] := \Pr_{x \in \{0,1\}^n}[f(x) \neq f(x \oplus e_i)],
\]
where $x \oplus e_i$ is obtained from $x$ by flipping the $i$th coordinate, and the probability is taken with respect to the uniform measure on $\{0,1\}^n$. The \emph{total influence} of the function is
\[
 I[f] := \sum_{i=1}^n I_i[f].
\]
Over the last thirty years, many results have been obtained on the influences of Boolean functions, and have proved extremely useful in such diverse fields as theoretical computer science, social choice theory and statistical physics, as well as in combinatorics (see, e.g., the survey~\cite{Kalai-Safra}).

\medskip

It is easy to see that the total influence of a function $f$ is none other than the size of the edge boundary of the set $A(f)=\{x\in \{0,1\}^n:\ f(x)=1\}$, appropriately normalised: viz., $I[f] = |\partial (A(f))|/2^{n-1}$. Hence, Corollary~\ref{cor:edge-iso} has the following reformulation in terms of Boolean functions and influences:
\begin{prop}[The weak edge isoperimetric inequality for $Q_n$ -- influence version]
\label{cor:edge-iso-inf}
If $f:\{0,1\}^n \rightarrow \{0,1\}$ is a Boolean function then
\begin{equation}\label{Eq:Main}
I[f] \geq 2\mathbb{E}[f]\log_2(1/\mathbb{E}[f]).
\end{equation}
\end{prop}
\noindent Theorem~\ref{thm:e} can be restated similarly.

\subsection{The biased measure on the discrete cube}

For $p\in [0,1]$, the {\em $p$-biased measure on $\p([n])$} is defined by
$$\mu_p^{(n)}(S) = p^{|S|} (1-p)^{n-|S|}\quad \forall S \subset [n].$$
In other words, we choose a random subset of $[n]$ by including each $j \in [n]$ independently with probability $p$. When $n$ is understood, we will omit the superscript $(n)$, writing $\mu_p = \mu_p^{(n)}$.

The definition of influences with respect to the biased measure is, naturally,
\[
I_i^p[f] := \Pr_{x \sim \mu_p}[f(x) \neq f(x \oplus e_i)],
\]
and $I^p[f] := \sum_{i=1}^n I_i^p[f]$. We abuse notation slightly and write $\mu_p(f):=\mathbb{E}_{\mu_p}[f]$. We remark that we may write $I^p[f] = \mu_p(\partial A(f))$, where we define the measure $\mu_p$ on subsets of $E(Q_n)$ by $\mu_p(\{x,x\oplus e_i\}) = p^{\sum_{j \neq i} x_i}(1-p)^{n-1-\sum_{j \neq i}x_j}$. (Note that $\mu_p(E(Q_n)) = n$, so $\mu_p$ is not a probability measure on $E(Q_n)$ unless $n=1$.)

\medskip

Many of the applications of influences (e.g., to the study of percolation \cite{BKS}, threshold phenomena in random graphs \cite{Bourgain99,Friedgut-SAT}, and hardness of approximation \cite{Dinur-Safra}) rely upon the use of the biased measure on the discrete cube. As a result, many of the central results on influences have been generalized to the biased setting (e.g. \cite{Friedgut98,Friedgut-Kalai,Hatami12}), and the edge isoperimetric inequality is no exception. The following `biased' generalization of Proposition~\ref{cor:edge-iso-inf} is considered folklore (see~\cite{KK06}).
\begin{thm}[The weak biased edge isoperimetric inequality for $Q_n$]
\label{thm:edge-iso-biased}
If $f:\{0,1\}^n \rightarrow \{0,1\}$ is a Boolean function, and $0 < p \leq 1/2$, then
\begin{equation}\label{Eq:Main}
pI^p[f] \geq \mu_p(f)\log_p(\mu_p(f)).
\end{equation}
The same statement holds for all $p \in (0,1)$ if $f$ is monotone increasing.
\end{thm}
\noindent Note that a function $f:\{0,1\}^n \to \{0,1\}$ is said to be {\em monotone increasing} if $f(x) \leq f(y)$ whenever $x_i \leq y_i$ for all $i \in [n]$. An easy inductive proof of Theorem \ref{thm:edge-iso-biased} is presented in~\cite{KK06}.

\subsection{A stability version of the biased edge isoperimetric inequality}

The first main result of this paper is the following stability version of Theorem~\ref{thm:edge-iso-biased}.
\begin{thm}
\label{thm:skewed-iso-stability} There exist absolute constants $c_0,C_1 >0$ such that the following holds. Let $0<p\leq\frac{1}{2}$, and let
$\epsilon\leq c_{0}/\ln(1/p)$. Let $f\colon\left\{ 0,1\right\} ^{n}\to\left\{ 0,1\right\} $
be a Boolean function such that
\[
pI^{p}[f]\leq\mu_{p}(f)\left(\log_{p}(\mu_{p}(f))+\epsilon\right).
\]
Then there exists a subcube $S\subset\{0,1\}^{n}$ such that
\begin{equation}
\mu_{p}(f\Delta1_{S})\leq C_1 \frac{\epsilon \ln(1/p)}{\ln\left(1/(\epsilon \ln(1/p)) \right)}\mu_{p}(f),\label{eq:conc-1}
\end{equation}
where $f \Delta 1_S := \{x:f(x) \neq 1_S(x)\}$.
\end{thm}

\noindent If we assume further that $f$ is monotone increasing, then the above theorem can be extended to $p>1/2$.
\begin{thm}
\label{thm:mon-iso-stability} For any $\eta>0$, there exist $C_{1}=C_{1}(\eta)$,
$c_{0}=c_{0}(\eta)>0$ such that the following holds. Let $0<p\leq1-\eta$,
and let $\epsilon\leq c_{0}/\ln(1/p)$. Let $f\colon\left\{ 0,1\right\} ^{n}\to\left\{ 0,1\right\} $
be a monotone increasing Boolean function such that
\[
pI^{p}[f]\leq\mu_{p}(f)\left(\log_{p}(\mu_{p}(f))+\epsilon\right).
\]
Then there exists a monotone increasing subcube $S\subset\{0,1\}^{n}$ such that
\begin{equation}
\mu_{p}(f\Delta1_{S})\leq C_1 \frac{\epsilon \ln(1/p)}{\ln\left(1/(\epsilon \ln(1/p)\right)}\mu_{p}(f).\label{eq:conc}
\end{equation}
\end{thm}
\noindent (Note subset $S \subset \{0,1\}^n$ is said to be {\em monotone increasing} if its indicator function is monotone increasing. The {\em indicator function} of $S \subset \{0,1\}^n$ is the Boolean function on $\{0,1\}^n$ taking the value $1$ on $S$ and $0$ outside $S$.)

As we show in Section~\ref{sec:examples}, Theorems \ref{thm:skewed-iso-stability} and \ref{thm:mon-iso-stability} are sharp, up to the values of the constants $c_0,C_1$, and this remains the case even if the subcube in the conclusion of Theorem \ref{thm:mon-iso-stability} is allowed to be non-monotone. Moreover, the dependence of $c_0,C_1$ on $\eta$ in Theorem \ref{thm:mon-iso-stability} cannot be removed --- though, for the sake of brevity, we do not attempt to optimise the dependence of these constants on $\eta$ in our proof.

The proofs of Theorems \ref{thm:skewed-iso-stability} and \ref{thm:mon-iso-stability} use induction on $n$, in a similar way to the proof of Theorem \ref{thm:e} in~\cite{Ellis}, but unlike in previous works, they do not use any Fourier-theoretic tools, relying only upon `elementary' (though intricate) combinatorial and analytic arguments.

\medskip

Theorems \ref{thm:skewed-iso-stability} and \ref{thm:mon-iso-stability} are crucial tools in a recent work of the authors~\cite{EKL16+}, which establishes
a general method for leveraging Erd\H{o}s-Ko-Rado type results in extremal combinatorics into strong stability versions,
without going into the proofs of the original results. This method is used in \cite{EKL16+} to obtain sharp (or almost-sharp) stability versions of the Erd\H{o}s-Ko-Rado theorem itself \cite{EKR}, of the seminal `complete intersection theorem' of Ahlswede and Khachatrian \cite{AK}, of Frankl's recent result on the Erd\H{o}s matching conjecture \cite{Frankl13}, of the Ellis-Filmus-Friedgut proof of the Simonovits-S\'{o}s conjecture \cite{EFF12}, and of various Erd\H{o}s-Ko-Rado type results on $r$-wise (cross)-$t$-intersecting families.

Theorem \ref{thm:mon-iso-stability} is also used in \cite{unions} by the first and last authors to obtain sharp upper bounds on the size of the union of several intersecting families of $k$-element subsets of $[n]$, where $k \leq (1/2-o(1))n$, extending results of Frankl and F\"uredi \cite{ff}.

\subsection{A biased version of the `full' edge isoperimetric inequality for monotone increasing families}

While the generalization of the `weak' edge isoperimetric inequality (i.e., Corollary~\ref{cor:edge-iso}) to the biased measure has
been known for a long time, such a generalization of the `full' edge isoperimetric inequality (i.e., Theorem~\ref{thm:edge-iso}) was hitherto unknown. In his talk at the 7th European Congress of Mathematicians \cite{Kalai16}, Kalai asked whether there is a natural generalization of Theorem~\ref{thm:edge-iso} to the measure $\mu_p$ for $p<1/2$.

We answer Kalai's question in the affirmative by showing that the most natural such generalization does not hold for arbitrary families, but does hold (even for $p>1/2$) under the additional assumption that the family is monotone increasing. (We say a family $\f \subset \pn$ is {\em monotone increasing} if $(S \in \f,\ S \subset T) \Rightarrow T \in \f$.)

In order to present our result, we first define the appropriate generalization of lexicographic families for the biased-measure setting. Note that while in the uniform measure ($p = 1/2$) case, for any $\f \subset \p([n])$ there exists a
lexicographic family $\l \subset \pn$ with the same measure as $\f$, this does not hold in general for $p \neq 1/2$. However, the situation can be remedied by passing to subsets of the Cantor space $\p(\mathbb{N})$. We let $\Sigma$ be the $\sigma$-algebra on $\p(\mathbb{N})$ generated by $\cup_{n \in \mathbb{N}} \p([n])$, and for each $p \in (0,1)$, we let $\mu^{(\mathbb{N})}_p$ be the natural $p$-biased measure on $(\p(\mathbb{N}),\Sigma)$ (the unique measure that `projects' to the measure $\mu_p^{(n)}$ on $\pn$, for each $n \in \mathbb{N}$). By analogy with subsets of $[n]$, if $\f \in \Sigma$ and $i \in \mathbb{N}$ we define the {\em $i$th influence of $\f$ w.r.t. $\mu^{(\mathbb{N})}_p$ by}
 $$I^p_i[\f] := \Pr_{S \sim \mu^{(\mathbb{N})}_p}[\f \cap \{S,S \Delta \{i\}\}|=1]$$
 and the {\em total influence of $\f$ w.r.t. $\mu^{(\mathbb{N})}_p$} by $I^p[\f] = \sum_{i=1}^{\infty}I_i^p[\f]$.

Just as for subsets of $[n]$, the {\em lexicographic ordering on $\p(\mathbb{N})$} is defined by $S > T$ iff $\min(S \Delta T) \in S$. For each $\lambda \in [0,1]$, we let $\l_{\lambda} \subset \p(\mathbb{N})$ be the unique initial segment of the lexicographic ordering on $\p(\mathbb{N})$ with $\mu_{1/2}^{(\mathbb{N})}(\l_{\lambda}) = \lambda$. (It is easily checked that initial segments of the lexicographic ordering on $\p(\mathbb{N})$ are $\Sigma$-measurable.) Moreover, the function $f_p:\ \lambda \mapsto \mu^{(\mathbb{N})}_p(\l_{\lambda})$ is continuous and monotone increasing, for each $p \in (0,1)$, with $f_p(0)=0$ and $f_p(1)=1$. Hence, by the intermediate value theorem, for any $p \in (0,1)$ and any $x \in [0,1]$, there exists $\lambda \in [0,1]$ such that $\mu^{(\mathbb{N})}_p(\l_{\lambda})=x$. In particular, for each $n \in \mathbb{N}$ and each $\f \subset \pn$, there exists $\lambda \in [0,1]$ such that $\mu^{(\mathbb{N})}_p(\l_{\lambda})=\mu^{(n)}_p(\f)$, where $\mu_p^{(n)}$ denotes the $p$-biased measure on $\pn$. We prove this family $\l_{\lambda}$ has total influence no larger than that of $\f$:

\begin{thm}
\label{thm:Monotone}Let $p\in\left(0,1\right)$, and let $\f\subset\pn$
be a monotone increasing family. Let $\lambda \in [0,1]$ be such that $\mu^{(\mathbb{N})}_p \left(\f\right)=\mu^{(n)}_{p}\left(\l_{\lambda}\right)$. Then $I^{p}\left[\f\right]\ge I^{p}\left[\l_{\lambda}\right]$. (Here, $I^p[\f]$ is defined in terms of the $p$-biased measure on $\p([n])$, whereas $I^p[\l_{\lambda}]$ is defined in terms of the $p$-biased measure on $(\p(\mathbb{N}),\Sigma)$.)
\end{thm}
Our proof uses the Kruskal-Katona theorem~\cite{Katona66,Kruskal63}, the Margulis-Russo Lemma \cite{Margulis,Russo}, and some additional analytic and combinatorial arguments.

In fact, Theorem \ref{thm:edge-iso} (the `full' edge-isoperimetric inequality of Harper, Bernstein, Lindsey and Hart) follows quickly from Theorem \ref{thm:Monotone}, via a monotonization argument, so our proof of Theorem \ref{thm:Monotone} provides a new proof of Theorem \ref{thm:edge-iso}, via the
Kruskal-Katona theorem. This may be of independent interest, and may be somewhat surprising, as the Kruskal-Katona theorem is more immediately connected to the vertex-boundary of an increasing family, than to its edge-boundary.

\medskip

We remark that the assertion of Theorem \ref{thm:Monotone}
is false for arbitrary (i.e., non-monotone) functions, for each value of $p \neq 1/2$. Indeed, it is easy to check that for each $p \in (0,1) \setminus \{\tfrac{1}{2}\}$, the `antidictatorship' $\mathcal{A} = \{S \subset [n]:\ 1 \notin S\}$ has $I^p[\mathcal{A}] = 1 < I^p[\l_{\lambda}]$, where $\lambda$ is such that $\mu_p(\l_{\lambda}) = 1-p\ (= \mu_p(\mathcal{A}))$. (See Remark \ref{remark:anti}.)

\subsection{Organization of the paper}

In Section \ref{sec:prelim}, we outline some notation and present an inductive proof of Theorem \ref{thm:edge-iso-biased}, some of whose ideas and components we will use in the sequel. In Section \ref{sec:main} (the longest part of the paper), we prove Theorems~\ref{thm:skewed-iso-stability} and~\ref{thm:mon-iso-stability}. In Section \ref{sec:examples}, we give examples showing that Theorems \ref{thm:skewed-iso-stability} and~\ref{thm:mon-iso-stability} are sharp (in a certain sense). In Section \ref{sec:KK}, we prove Theorem \ref{thm:Monotone} and show how to use it to deduce Theorem \ref{thm:edge-iso}. We conclude the paper with some open problems in Section~\ref{sec:open}.

\section{An inductive proof of Theorem~\ref{thm:edge-iso-biased}}
\label{sec:prelim}

In this section, we outline some notation and terminology, and present a simple inductive proof of Theorem~\ref{thm:edge-iso-biased}; components and ideas from this proof will be used in the proofs of Theorems \ref{thm:skewed-iso-stability} and \ref{thm:mon-iso-stability}.

\subsection{Notation and terminology}

When the `bias' $p$ (of the measure $\mu_p$) is clear from the context (including throughout Sections~\ref{sec:prelim} and~\ref{sec:main}), we will sometimes omit it from our notation, i.e. we will sometimes write $\mu(f) := \mu_p(f)$ and $I[f]:=I^{p}[f]$. Moreover, when the Boolean function $f$ is clear from the context, we will sometimes omit it from our notation, i.e. we will sometimes write $\mu := \mu(f)$, $I := I[f]$ and $I_i : = I_i[f]$. If $S \subset \{0,1\}^n$, we write $1_{S}$ for its indicator function, i.e. the Boolean function on $\{0,1\}^n$ taking the value $1$ on $S$ and $0$ outside $S$. A {\em dictatorship} is a Boolean function $f:\{0,1\}^n \to \{0,1\}$ of the form $f = 1_{\{x_j=1\}}$ for some $j \in [n]$; an {\em antidictatorship} is one of the form $f = 1_{\{x_j=0\}}$. Abusing notation slightly, we will sometimes identify a family $\mathcal{F} \subset \mathcal{P}([n])$ with the corresponding indicator function $1_{\{x \in \{0,1\}^n:\ \{i \in [n]:\ x_i=1\} \in \mathcal{F}\}}$.

A {\em subcube} of $\{0,1\}^n$ is a set of the form $\{x \in \{0,1\}^n:\ x_i = a_i\ \forall i \in F\}$, where $F \subset [n]$ and $a_i \in \{0,1\}$ for all $i \in F$; $F$ is called the set of {\em fixed coordinates} of the subcube.

We use the convention $0\log_{p}(0)=0$ (for all $p \in (0,1)$); this turns $x\mapsto x\log_{p}(x)$ into a continuous function on $[0,1]$. If $S$ and $T$ are sets, we write $S \subset T$ if $S$ is a (not necessarily proper) subset of $T$.

If $f:\{0,1\}^n \to \{0,1\}$ and $i \in [n]$, we define the function $f_{i\to0}:\{0,1\}^{[n]\setminus\{i\}}\to\{0,1\}$ by $f_{i\mapsto0}(y)=f(x)$, where $x_{i}=0$ and $x_{j}=y_{j}$ for all $j\in[n]\setminus\{i\}$. In other words, $f_{i\to0}$ is the restriction of $f$ to the lower half-cube $\{x\in\{0,1\}^{n}:x_{i}=0\}$. We define $f_{i\to1}$ similarly. For brevity, we will often write
\begin{align*}
\mu_i^- & = \mu_{i}^{-}(f) :=\mu_{p}(f_{i\to0}),\\
\mu_i^+ & = \mu_{i}^{+}(f)  :=\mu_{p}(f_{i\to1}),\\
I_i^- & = I_{i}^{-}[f] :=I^{p}[f_{i\to0}],\\
I_i^+ & = I_{i}^{+}[f] :=I^{p}[f_{i\to1}].
\end{align*}
Note that
\begin{equation}
p\mu_{i}^{+}(f)+(1-p)\mu_{i}^{-}(f)=\mu(f) \label{eq:basic1}
\end{equation}
and that
\begin{equation}
I\left[f\right]=I_{i}\left[f\right]+pI_{i}^{+}\left[f\right]+\left(1-p\right)I_{i}^{-}\left[f\right].\label{eq:Basic 2}
\end{equation}

\subsection{A proof of Theorem~\ref{thm:edge-iso-biased}}

The proof uses induction on $n$ together with equations (\ref{eq:basic1}) and (\ref{eq:Basic 2}), and the following technical lemma.
\begin{lem}
\label{Lem:ind-step}Let $p\in\left(0,1\right)$, and let $F,G,H\colon\left[0,1\right]\times\left[0,1\right]\to\left[0,\infty\right)$
be the functions defined by
\begin{align*}
F\left(x,y\right) & =px\log_{p}x+\left(1-p\right)y\log_{p}y+px-py,\\
G\left(x,y\right) & =\left(px+\left(1-p\right)y\right)\log_{p}\left(\left(px+\left(1-p\right)y\right)\right),\\
H\left(x,y\right) & =px\log_{p}x+\left(1-p\right)y\log_{p}y+py-px.
\end{align*}

\begin{enumerate}
\item If $x\ge y\ge0$, then $F\left(x,y\right)\ge G\left(x,y\right)$.
\item If $y\ge x\ge0$ and $p\le\frac{1}{2}$, then $H\left(x,y\right)\ge G\left(x,y\right)$.
\end{enumerate}
\end{lem}

\begin{proof}[Proof of Lemma \ref{Lem:ind-step}]
Clearly, for all $y \geq 0$ we have $F\left(y,y\right)=G\left(y,y\right)=H\left(y,y\right)$, and for all $x,y\geq0$, we have
\begin{align}
\frac{\partial F}{\partial x} & =p\log_{p}x+\frac{p}{\ln p}+p=p\log_{p}(px)+\frac{p}{\ln p},\nonumber \\
\frac{\partial G}{\partial x} & =p\log_{p}\left(px+\left(1-p\right)y\right)+\frac{p}{\ln p},\label{eq:partial-devs} \nonumber\\
\frac{\partial H}{\partial y} & =\left(1-p\right)\log_{p}y+p+\frac{1-p}{\ln p}\\
& =\left(1-p\right)\log_{p}(\left(1-p\right)y)-\left(1-p\right)\log_{p}\left(1-p\right)+p + \frac{1-p}{\ln p}\nonumber, \\
\frac{\partial G}{\partial y} & =\left(1-p\right)\log_{p}\left(px+\left(1-p\right)y\right)+\frac{1-p}{\ln p}.\nonumber
\end{align}
Clearly, we have $\frac{\partial F(x,y)}{\partial x}\geq\frac{\partial G(x,y)}{\partial x}$
for all $x,y\ge0$, and therefore $F\left(x,y\right)\ge G\left(x,y\right)$
for all $x\ge y\ge0$, proving (1). We assert that similarly, $\frac{\partial H(x,y)}{\partial y}\ge\frac{\partial G(x,y)}{\partial y}$
for all $x,y\geq0$, if $p \leq 1/2$. (This will imply (2).) Indeed,
\begin{align*}
\frac{\partial H}{\partial y} & =\left(1-p\right)\log_{p}(\left(1-p\right)y)-\left(1-p\right)\log_{p}\left(1-p\right)+p + \frac{1-p}{\ln p}\\
 & \ge\frac{\partial G}{\partial y}+p-\left(1-p\right)\log_{p}\left(1-p\right).
\end{align*}
Hence, it suffices to prove the following.
\begin{claim}
\label{claim on alpha(p)}Define $K:(0,1) \to \mathbb{R};\ K(p)=p-\left(1-p\right)\log_{p}\left(1-p\right)$.
Then $K\left(p\right)>0$ for all $p\in(0,\frac{1}{2})$, $K(1/2)=0$ and $K(p) < 0$ for all $p \in (1/2,1)$. \end{claim}
\begin{proof}[Proof of Claim \ref{claim on alpha(p)}]
Clearly, we have $K(1/2)=0$. It suffices to show that
$$\alpha(p):=K\left(p\right)\ln(1/p) = -p\ln p+(1-p)\ln(1-p)$$
is positive for all $p\in(0,1/2)$, since $\alpha(1-p)=-\alpha(p)$ for all $p \in (0,1)$. Note that $\alpha(x) \to 0$ as $x \to 0$ and that $\alpha(x) \to 0$ as $x \to 1$, so we may extend $\alpha$ to a continuous function on $[0,1]$ by defining $\alpha(0) = \alpha(1)=0$.

We have
$$\alpha'(x) = -\ln x - \ln(1-x) - 2.$$

Suppose for a contradiction that $\alpha$ has a zero in $(0,1/2)$. Then, since $0$ and $1/2$ are also zeros of $\alpha$, $\alpha$ would have at least two stationary points in $(0,1/2)$. This cannot occur, because $\alpha'(x)=0$ implies $x(1-x)=e^{-2}$, which has at most one solution in $(0,1/2)$, since if $x_0$ is a solution then $1-x_0$ is also solution, and any quadratic equation has at most two solutions. Hence, $\alpha$ has no zeros in $(0,1/2)$. Since $\alpha'(x) \to \infty$ as $x \to 0$, we must have $\alpha(x) >0$ for all $x \in (0,1/2)$, as required.
\end{proof}
This completes the proof of Lemma~\ref{Lem:ind-step}.
\end{proof}

We can now prove Theorem \ref{thm:edge-iso-biased}.

\begin{proof}[Proof of Theorem~\ref{thm:edge-iso-biased}.]
It is easy to check that the theorem holds for $n=1$. Let $n \geq 2$, and suppose the statement of the theorem holds when $n$ is replaced by $n-1$. Let $f:\{0,1\}^n \to \{0,1\}$. Choose any $i \in [n]$. We split into two cases.
\begin{caseenv}
\item[Case (a)] $\mu_{i}^{-}\le\mu_{i}^{+}$.
\end{caseenv}

Applying the induction hypothesis to the functions $f_{i\to0}$
and $f_{i\to1}$, and using the fact that $I_{i}[f] \ge\mu_{i}^{+}-\mu_{i}^{-}$, we obtain
\begin{align*}
pI & =(1-p)pI_{i}^{-}[f]+p^{2}I_{i}^{+}[f]+pI_{i}[f]\\
 & \ge(1-p)\mu_{i}^{-} \log_{p}(\mu_{i}^{-}) +p\mu_{i}^{+} \log_{p}(\mu_{i}^{+})+p\left(\mu_{i}^{+}-\mu_{i}^{-}\right)\\
 & =F\left(\mu_i^{+},\mu_{i}^{-}\right)\ge G\left(\mu_{i}^{+},\mu_{i}^{-}\right)=\mu\log_{p}\left(\mu\right),
\end{align*}
 where $F$ and $G$ are as defined in Lemma \ref{Lem:ind-step}.
\begin{caseenv}
\item[Case (b)] $\mu_{i}^{-}\ge\mu_{i}^{+}$.
\end{caseenv}

The proof in this case is similar: applying the induction hypothesis to the functions $f_{i\to0}$
and $f_{i\to1}$, and using the fact that $I_{i}[f] \ge\mu_{i}^{-}-\mu_{i}^{+}$, we obtain
\begin{align*}
pI & =(1-p)pI_{i}^{-}[f]+p^{2}I_{i}^{+}[f]+pI_{i}[f]\\
 & \ge(1-p)\mu_{i}^{-}\log_{p}(\mu_{i}^{-})+p\mu_{i}^{+}\log_{p}(\mu_{i}^{+})+p\left(\mu_{i}^{-}-\mu_{i}^{+}\right)\\
 & =H\left(\mu^{+},\mu_{i}^{-}\right)\ge G\left(\mu_{i}^{+},\mu_{i}^{-}\right)=\mu\log_{p}\left(\mu\right),
\end{align*}
using the fact that $p \leq 1/2$.
\end{proof}

We remark that the above proof shows that if $f$ is monotone increasing, then the statement of Theorem~\ref{thm:edge-iso-biased} holds for all $p \in (0,1)$. (Indeed, if $f$ is monotone increasing, then $\mu_{i}^{-}\le\mu_{i}^{+}$ for all $i\in\left[n\right]$, so the assumption $p \leq 1/2$ is not required.)

\section{Proofs of the `biased' isoperimetric stability theorems}
\label{sec:main}

In this section, we prove Theorems \ref{thm:skewed-iso-stability} and \ref{thm:mon-iso-stability}. As the proofs of the two theorems follow the same strategy, we present them in parallel.

The proof of Theorem~\ref{thm:skewed-iso-stability} (and similarly, of Theorem~\ref{thm:mon-iso-stability}) consists of five steps. Assume that $f$ satisfies the assumptions of the theorem.
\begin{enumerate}
\item We show that for each $i \in [n]$, either $I_{i}[f]$ is small or else $\min\left\{ \mu_{i}^{-},\mu_{i}^{+}\right\}$
is `somewhat' small. In other words, the influences of $f$ are similar to the influences of a subcube.
\item We show that $\mu$ must be either very close to 1 or `fairly' small, i.e., bounded away from 1 by a constant. (In the proof of Theorem~\ref{thm:mon-iso-stability}, the constant may depend on $\eta$.)
\item We show that unless $\mu$ is very close to 1, there exists $i \in [n]$ such
that $I_{i}[f]$ is large. This implies that $\min\{\mu_{i}^{-},\mu_{i}^{+}\}$ is `somewhat' small.
\item We prove two `bootstrapping' lemmas saying that if $\mu_{i}^{-}$ is `somewhat'
small, then it must be `very' small, and that if $\mu_{i}^{+}$ is `somewhat'
small, then it must be `very' small. This implies that $f$ is `very' close to being contained in a dictatorship or an antidictatorship.
\item Finally, we prove each theorem by induction on $n$.
\end{enumerate}
\medskip

\noindent From now on, we let $f\colon\left\{ 0,1\right\} ^{n}\to\left\{ 0,1\right\}$ such that $pI^{p}[f]\leq\mu_{p}(f)(\log_{p}(\mu_{p}(f))+\epsilon)$. By reducing $\epsilon$ if necessary, we may assume that $pI^{p}[f] = \mu_{p}(f)(\log_{p}(\mu_{p}(f))+\epsilon)$, i.e., using the more compact notation outlined above, $pI[f] = \mu (\log_p(\mu) + \epsilon)$.

\subsection{Relations between the influences of $f$ and the influences of its
restrictions $f_{i\to0},f_{i\to1}$}
\noindent We define $\epsilon_{i}^{-},\epsilon_{i}^{+}$ by
\begin{align*}
pI_{i}^{-}=\mu_{i}^{-}\left(\log_{p}(\mu_{i}^{-})+\epsilon_{i}^{-}\right),\qquad pI_{i}^{+}=\mu_{i}^{+}\left(\log_{p}(\mu_{i}^{+})+\epsilon_{i}^{+}\right).
\end{align*}
Note that Theorem \ref{thm:edge-iso-biased} implies that $\epsilon_{i}^{-},\epsilon_{i}^{+}\geq 0$. We define the functions $F,G,H,K$ as in the proof of Theorem \ref{thm:edge-iso-biased}.

We would now like to express the fact that $I[f]$ is
small in terms of $\epsilon_{i}^{-},\epsilon_{i}^{+},\mu_{i}^{-},\mu_{i}^{+}$.
For each $i\in\left[n\right]$ such that $\mu_{i}^{-}\le\mu_{i}^{+}$, we have
\begin{align}
\begin{split}\mu(\log_{p}(\mu)+\epsilon) & =pI[f]=(1-p)pI_{i}^{-}+p^{2}I_{i}^{+}+pI_{i}[f]\\
 & =(1-p)\mu_{i}^{-}\left(\log_{p}(\mu_{i}^{-})+\epsilon_{i}^{-}\right)+p\mu_{i}^{+}\left(\log_{p}(\mu_{i}^{+})+\epsilon_{i}^{+}\right)\\
 & +p(\mu_{i}^{+}-\mu_{i}^{-})+p(I_{i}[f]-\mu_{i}^{+}+\mu_{i}^{-});
\end{split}
\label{eq:relation}
\end{align}
rearranging (\ref{eq:relation}) gives
\begin{align}
\epsilon_{i}': & =\mu\epsilon-p\mu_{i}^{+}\epsilon_{i}^{+}-\left(1-p\right)\mu_{i}^{-}\epsilon_{i}^{-}\nonumber \\
& = p\mu_i^+(\epsilon-\epsilon_i^+) + (1-p)\mu_i^-(\epsilon-\epsilon_i^-) \nonumber\\
 & = F\left(\mu_{i}^{+},\mu_{i}^{-}\right)-G\left(\mu_{i}^{+},\mu_{i}^{-}\right)+p\left(I_{i}[f]-(\mu_{i}^{+}-\mu_{i}^{-})\right).\label{eq:rearranged-1}
\end{align}
Similarly, for each $i\in [n]$ such that $\mu_{i}^{-}\ge\mu_{i}^{+}$, we have
\begin{align}
\epsilon_{i}' & :=\mu\epsilon-p\mu_{i}^{+}\epsilon_{i}^{+}-\left(1-p\right)\mu_{i}^{-}\epsilon_{i}^{-}\nonumber \\
& = p\mu_i^+(\epsilon-\epsilon_i^+) + (1-p)\mu_i^-(\epsilon-\epsilon_i^-) \nonumber\\
 & = H\left(\mu_{i}^{+},\mu_{i}^{-}\right)-G\left(\mu_{i}^{+},\mu_{i}^{-}\right)+p\left(I_{i}[f]-(\mu_{i}^{-}-\mu_{i}^{+})\right).\label{eq:rearranged-2}
\end{align}
This allows us to deduce two facts about the structure of $f$.
\begin{itemize}
\item By Lemma \ref{Lem:ind-step}, we have $\epsilon_{i}'\geq0$ for all $i \in [n]$. This implies that
either $\epsilon_{i}^{+}\le\epsilon$ or $\epsilon_{i}^{-}\le\epsilon$.
Together with the induction hypothesis, this will imply (in Section \ref{subsec:ind}) that either $f_{i\to0}$
or $f_{i\to1}$ is structurally close to a subcube.
\item $F\left(\mu_{i}^{+},\mu_{i}^{-}\right)-G\left(\mu_{i}^{+},\mu_{i}^{-}\right)$
(resp. $H\left(\mu_{i}^{+},\mu_{i}^{-}\right)-G\left(\mu_{i}^{+},\mu_{i}^{-}\right)$)
is small whenever $\mu_{i}^{+}\ge\mu_{i}^{-}$ (resp. $\mu_{i}^{-}\ge\mu_{i}^{+}$).
Note that the proof of Lemma \ref{Lem:ind-step} shows that whenever
$\mu_{i}^{+}\ge\mu_{i}^{-}$ (resp. $\mu_{i}^{-}\ge\mu_{i}^{+}$)
then $F\left(\mu_{i}^{+},\mu_{i}^{-}\right)$ (resp. $H\left(\mu_{i}^{+},\mu_{i}^{-}\right)$)
is equal to $G\left(\mu_{i}^{+},\mu_{i}^{-}\right)$ only if $\mu_{i}^{+}=\mu_{i}^{-}$
or $\mu_{i}^{-}=0$. We will later show (in Claims \ref{claim:2-cases-constant-p}-\ref{claim:2-cases-small-p}) that if $F\left(\mu_{i}^{+},\mu_{i}^{-}\right)$ (resp. $H\left(\mu_{i}^{+},\mu_{i}^{-}\right)$) is approximately
equal to $G\left(\mu_{i}^{+},\mu_{i}^{-}\right)$, then either $\min\left\{ \mu_{i}^{-},\mu_{i}^{+}\right\} $
is small or else $I_{i}[f]$ is small.
\end{itemize}
The following lemma will be used to relate $\mu_i^{+}$ and $\mu_i^-$ to $F(\mu_{i}^{+},\mu_{i}^{-})-G(\mu_{i}^{+},\mu_{i}^{-})$ (or to
$H(\mu_{i}^{+},\mu_{i}^{-})-G(\mu_{i}^{+},\mu_{i}^{-})$), in a more convenient
way.
\begin{lem}
\label{lemma:analysis-basic-functions} If $0 < p < 1$ and $x\ge y\ge0$, then
\[
F\left(x,y\right)-G\left(x,y\right) \geq p\left(x-y\right)\log_{p}\left(\frac{px}{px+\left(1-p\right)y}\right).
\]
If $0 < p\le \tfrac{1}{2}$ and $y\ge x\ge0$, then
\[
H\left(x,y\right)-G\left(x,y\right) \geq \left(1-p\right)\left(y-x\right)\log_{p}\left(\frac{\left(1-p\right)y}{px+\left(1-p\right)y}\right).
\]
If $0 < p\le e^{-2}$ and $y\ge x \ge0$, then
\[
H\left(x,y\right)-G\left(x,y\right) \geq \tfrac{1}{2}p \left(y-x\right).
\]
\end{lem}

\begin{proof}
We show that
\begin{align}
\left.\frac{\partial}{\partial u}\left(F\left(u,y\right)-G\left(u,y\right)\right)\right|_{u=t} & \ge p\log_{p}\left(\frac{px}{px+\left(1-p\right)y}\right)\ \forall y \leq t \leq x, \ 0 < p < 1, \label{eq: calc(1)}\\
\left. \frac{\partial}{\partial u}\left(H\left(x,u\right)-G\left(x,u\right)\right)\right|_{u=t} & \ge\left(1-p\right)\log_{p}\left(\frac{\left(1-p\right)y}{px+\left(1-p\right)y}\right) \ \forall x \leq t \leq y,\ 0 < p \leq 1/2, \label{eq:calc(2)}\\
\left.\frac{\partial}{\partial u}\left(H\left(x,u\right)-G\left(x,u\right)\right)\right|_{u=t} & \ge \tfrac{1}{2}p, \ \forall x \leq t \leq y,\ 0 < p \leq e^{-2}.\label{eq:calc(3)}
\end{align}
These inequalities will complete the proof of the lemma,
by the Fundamental Theorem of Calculus.

Using (\ref{eq:partial-devs}), we have
\begin{align*}
\left. \frac{\partial }{\partial u}\left(F\left(u,y\right)-G\left(u,y\right)\right)\right|_{u=t} & =p\log_{p}\left(pt\right)+\frac{p}{\ln\left(p\right)}-p\log_{p}\left(\left(1-p\right)y+pt\right)-\frac{p}{\ln\left(p\right)}\\
 & =p\log_{p}\left(\frac{pt}{\left(1-p\right)y+pt}\right)\ge p\log_{p}\left(\frac{px}{(1-p)y+px}\right),
\end{align*}
 proving (\ref{eq: calc(1)}). Similarly, if $p \leq 1/2$ and $x \leq t \leq y$, then
\begin{align*}
\left.\frac{\partial }{\partial u}\left(H\left(x,u\right)-G\left(x,u\right)\right)\right|_{u=t}  & = \left(1-p\right)\log_{p}\left(\frac{\left(1-p\right)t}{px+\left(1-p\right)t}\right)+K\left(p\right)\\
& \ge\left(1-p\right)\log_{p}\left(\frac{\left(1-p\right)y}{px+\left(1-p\right)y}\right)+K\left(p\right)\\
 & \ge\left(1-p\right)\log_{p}\left(\frac{\left(1-p\right)y}{px+\left(1-p\right)y}\right),
\end{align*}
proving (\ref{eq:calc(2)}). It is easy to check that for all $p\le e^{-2}$, we have $K\left(p\right)\ge\frac{p}{2}$. Hence, if $x \leq t \leq y$ and $0 < p \leq e^{-2}$, then
\[
\left. \frac{\partial }{\partial u}\left(H\left(x,u\right)-G\left(x,u\right)\right)\right|_{u=t} \ge K\left(p\right)\ge\tfrac{p}{2},
\]
proving (\ref{eq:calc(3)}).
\end{proof}

\subsection{Either $I_{i}[f]$ is small, or $\min\left\{ \mu_{i}^{-},\mu_{i}^{+}\right\} $
is small}

We now show that the influences of $f$ are similar to the influences of a subcube. Note that if $f=1_{S}$ for a subcube $S = \{x \in \{0,1\}^n:\ x_i = a_i\ \forall i \in T\}$, where $T \subset [n]$ and $a_i \in \{0,1\}$ for all $i \in T$, then $\min\left\{ \mu_{i}^{-},\mu_{i}^{+}\right\} =0$
for each $i\in T$, and $I_{i}[f]=0$ for each $i\notin T$. We prove that an approximate version of this statement holds, under our hypotheses.

We start with the simplest case, which is $\zeta<p\le\frac{1}{2}$ for some $\zeta >0$.
\begin{claim}
\label{claim:2-cases-constant-p} Let $\zeta >0$. There exists $C_{2}=C_{2}(\zeta)>0$ such that if $\zeta \leq p \leq 1/2$, then for each $i\in[n]$, one of the following holds.
\begin{description}
\item [{Case (1)}] \label{case:small-influence-1-1} We have $I_{i}[f]\leq C_{2}\epsilon_{i}'$,
and $\min\left\{ \mu_{i}^{-},\mu_{i}^{+}\right\} \geq\left(1-C_{2}\epsilon\right)\mu$.
\item [{Case (2)}] \label{case:large-influence-1-1} We have $\min\left\{ \mu_{i}^{-},\mu_{i}^{+}\right\} \le C_{2}\epsilon_{i}'$,
and $I_{i}[f] \geq\left(1-C_{2}\epsilon\right)\mu$.
\end{description}
\end{claim}
We remark that in Claim \ref{claim:2-cases-constant-p}, it is necessary that $C_2$ depend on $\zeta$; this is evidenced e.g.\ by the function $f=1_{B}$ in Section \ref{sec:examples}, with $t=1$, $s=3$ and $i=2$.
\begin{proof}[Proof of Claim \ref{claim:2-cases-constant-p}.]
 By Lemma \ref{lemma:analysis-basic-functions} and (\ref{eq:rearranged-1}), if $\mu_i^-\leq \mu_i^+$ then
\begin{align}
p\left(\mu_{i}^{+}-\mu_{i}^{-}\right)\log_{p}\left(\frac{p\mu_{i}^{+}}{\mu}\right) & \le F\left(\mu_{i}^{+},\mu_{i}^{-}\right)-G\left(\mu_{i}^{+},\mu_{i}^{-}\right) \nonumber\\
& \leq\epsilon_{i}'-pI_{i}[f]-p\mu_{i}^{-}+p\mu_{i}^{+}.\label{eq:constant 2-cases(1)}
\end{align}
By Lemma \ref{lemma:analysis-basic-functions} and (\ref{eq:rearranged-2}), if $\mu_i^-\geq \mu_i^+$ then
\begin{align}
\left(1-p\right)\left(\mu_{i}^{-}-\mu_{i}^{+}\right)\log_{p}\left(\frac{\left(1-p\right)\mu_{i}^{-}}{\mu}\right) &\le H\left(\mu_{i}^{+},\mu_{i}^{-}\right)-G\left(\mu_{i}^{+},\mu_{i}^{-}\right) \nonumber\\
& \leq\epsilon_{i}'-pI_{i}[f]-p\mu_{i}^{+}+p\mu_{i}^{-}.\label{eq:constant 2-cases(2)}
\end{align}

Since the right-hand sides of (\ref{eq:constant 2-cases(1)}) and (\ref{eq:constant 2-cases(2)}) are non-negative, we have
\begin{equation}
I_{i}[f] -\left|\mu_{i}^{+}-\mu_{i}^{-}\right|\le\tfrac{1}{p} \epsilon_{i}'\leq \tfrac{1}{\zeta} \epsilon_{i}'.\label{eq:Inf_i is like monotone}
\end{equation}
We now split into two cases.

\begin{caseenv}
\item[Case (a):] $\min\left\{ \mu_{i}^{-},\mu_{i}^{+}\right\} \ge\frac{\mu}{2}$.
\end{caseenv}

In this case, we have
$$p\log_{p}\left(\frac{p\mu_{i}^{+}}{\mu}\right) = \Omega_{\zeta}(1),\quad \left(1-p\right)\log_{p}\left(\frac{\left(1-p\right)\mu_{i}^{-}}{\mu}\right) = \Omega_{\zeta}(1),$$
so
\[
\left|\mu_{i}^{+}-\mu_{i}^{-}\right| = O_{\zeta}\left(\epsilon'_{i}\right),
\]
by (\ref{eq:constant 2-cases(1)}) and (\ref{eq:constant 2-cases(2)}).
Equation (\ref{eq:Inf_i is like monotone}) now implies that $I_{i}[f] = O_{\zeta}\left(\epsilon'_{i}\right)$. Therefore, $\min\left\{ \mu_{i}^{-},\mu_{i}^{+}\right\} \ge\mu-I_{i}[f] = \mu-O_{\zeta}(\epsilon_{i}^{'}) = \mu - O_{\zeta}(\epsilon \mu) = \mu(1-O_{\zeta}(\epsilon))$. (Note that, by the definition of $\epsilon_i'$ in (\ref{eq:rearranged-1}) and(\ref{eq:rearranged-2}), we always have $\epsilon_i' \leq \epsilon \mu$.) Hence, Case (1) of the claim occurs.

\begin{caseenv}
\item[Case (b):] $\min\left\{ \mu_{i}^{-},\mu_{i}^{+}\right\} \le\frac{\mu}{2}$.
\end{caseenv}

Firstly, suppose in addition that $\mu_i^- \leq \mu_i^+$, so that $\mu_i^- \leq \mu/2$. Then $p(\mu_{i}^{+}-\mu_{i}^{-}) \geq p(\mu-\mu_i^-) \geq p\mu/2 \geq \zeta \mu/2 = \Omega_{\zeta}\left(\mu\right)$, so (\ref{eq:constant 2-cases(1)}) implies that
$$\log_{p}\left(\frac{p\mu_{i}^{+}}{\mu}\right) = O_{\zeta}( \epsilon_{i}^{'} / \mu).$$

Hence, $\ln\left(\frac{\mu}{p\mu_{i}^{+}}\right) = O_{\zeta}(\epsilon_{i}'/\mu)$, and therefore
\[
1+\frac{\left(1-p\right)\mu_{i}^{-}}{p\mu_{i}^{+}}=\frac{\mu}{p\mu_{i}^{+}} = \exp\left(O_{\zeta}(\epsilon_{i}^{'}/\mu)\right)=1+O_{\zeta}(\epsilon_{i}^{'}/\mu).
\]
 Therefore, $\mu_{i}^{-} = O_{\zeta}(\epsilon_{i}^{'})\frac{p \mu_{i}^{+}}{(1-p)\mu} = O_{\zeta}(\epsilon_{i}^{'})$.
We now have $I_{i}[f]\ge\mu-\mu_{i}^{-} = \mu - O_{\zeta}(\epsilon_i') = \mu - O_{\zeta}(\epsilon \mu) = (1-O_{\zeta}(\epsilon))\mu$. Hence, Case (2) of the claim occurs.

\medskip

Secondly, suppose in addition that $\mu_i^+ \leq \mu_i^-$, so that $\mu_i^+ \leq \mu/2$. Then we have $(1-p)(\mu_{i}^{-}-\mu_{i}^{+}) = \Omega\left(\mu\right)$, so (\ref{eq:constant 2-cases(2)}) implies that
$$\log_{p}\left(\frac{(1-p)\mu_{i}^{-}}{\mu}\right) = O( \epsilon_{i}^{'} / \mu).$$
Hence, $\ln\left(\frac{\mu}{(1-p)\mu_{i}^{-}}\right) = O_{\zeta}(\epsilon_{i}'/\mu)$, and therefore
\[
1+\frac{p\mu_{i}^{+}}{(1-p)\mu_{i}^{-}}=\frac{\mu}{(1-p)\mu_{i}^{-}} = \exp\left(O_{\zeta}(\epsilon_{i}^{'}/\mu)\right)=1+O_{\zeta}(\epsilon_{i}^{'}/\mu).
\]
 Therefore, $\mu_{i}^{+} = O_{\zeta}(\epsilon_{i}^{'})\frac{(1-p) \mu_{i}^{-}}{p \mu} = O_{\zeta}(\epsilon_{i}^{'})$. It follows that $I_{i}[f]\geq (1-O_{\zeta}(\epsilon))\mu$, so again, Case (2) of the claim must occur.

\end{proof}
We now prove a version of Claim \ref{claim:2-cases-constant-p} for monotone increasing $f$ and for all $p$ bounded away from 1. The idea of the proof
is the same, but the details are slightly messier, mainly because $p$ is no longer bounded away from $0$.
\begin{claim}
\label{claim:mon-2-cases} For any $\eta >0$, there exists $C_2 = C_2(\eta)>0$ such that the following holds. Suppose that $f$ is monotone increasing and that $0 < p \leq 1-\eta$. Let $i\in[n]$. Then one of the following
must occur.
\begin{description}
\item [{Case} (1)] We have $pI_{i}[f] \leq C_{2}\epsilon_{i}'\ln(1/p)$,
and $\mu_{i}^{-}\geq\left(1-C_{2}\epsilon\ln(1/p)\right)\mu$.
\item [{Case} (2)] We have $\mu_{i}^{-}\le C_{2}\epsilon_{i}'\ln(1/p)$,
and $pI_{i}[f] \geq\left(1-C_{2}\epsilon\ln(1/p)\right)\mu$.
\end{description}
\end{claim}

We remark that in Claim \ref{claim:mon-2-cases}, it is necessary that $C_2$ depend on $\eta$; this is evidenced e.g.\ by the function $f=1_{B}$ in Section \ref{sec:examples}, with $t=1$, $s=3$ and $i=1$.

\begin{proof}
By Lemma \ref{lemma:analysis-basic-functions} and equation (\ref{eq:rearranged-1}), we have
\[
pI_{i}[f] \log_{p}\left(\frac{p\mu_{i}^{+}}{\mu}\right) = p(\mu_i^+-\mu_i^-) \log_{p}\left(\frac{p\mu_{i}^{+}}{\mu}\right)\le F\left(\mu_{i}^{+},\mu_{i}^{-}\right)-G\left(\mu_{i}^{+},\mu_{i}^{-}\right)\leq\epsilon_{i}'.
\]
 We now split into two cases.
\begin{caseenv}
\item[Case (a):] $\mu_{i}^{+}\leq(1-\tfrac{\eta}{2})\frac{\mu}{p}$.
\end{caseenv}

If $\mu_{i}^{+}\leq(1-\tfrac{\eta}{2})\frac{\mu}{p}$, then Case (1)
of Claim \ref{claim:mon-2-cases} must occur, provided we take $C_{2}$
to be sufficiently large. Indeed, we then have
\[
pI_{i}[f] \log_{p}(1-\tfrac{\eta}{2}) = p(\mu_i^+-\mu_i^-) \log_{p}(1-\tfrac{\eta}{2})\leq p(\mu_i^+-\mu_i^-)\log_{p}\left(\frac{p\mu_{i}^{+}}{\mu}\right)\leq\epsilon_{i}',
\]
 which gives $pI_{i}[f] \leq\frac{1}{\ln\left(\frac{2}{2-\eta}\right)}\epsilon_{i}'\ln(1/p)\leq C_{2}\epsilon_{i}'\ln(1/p)$,
provided we choose $C_{2}\geq1/(\ln(2/(2-\eta)))$. This in turn implies that
\[
\mu_{i}^{-}=\mu-p(\mu_{i}^{+}-\mu_{i}^{-})=\mu-pI_{i}[f]\geq\mu-C_{2}\epsilon_{i}'\ln(1/p)\geq\mu-C_{2}\epsilon\mu\ln(1/p),
\]
 so Case (1) occurs, as asserted.
\begin{caseenv}
\item[Case (b):] $\mu_{i}^{+}\ge\left(1-\frac{\eta}{2}\right)\frac{\mu}{p}$.
\end{caseenv}

If $\mu_{i}^{+} \geq (1-\tfrac{\eta}{2})\frac{\mu}{p}$, then Case (2)
of Claim \ref{claim:mon-2-cases} must occur. Indeed, since $\mu_{i}^{-}\leq\mu$,
we have $pI_{i}[f]=p(\mu_{i}^{+}-\mu_{i}^{-})\geq\left(1-\tfrac{\eta}{2}-p\right)\mu\geq\tfrac{1}{2}\eta\mu$.
We now have
\[
\log_{p}\left(\frac{p\mu_{i}^{+}}{\mu}\right)\le\frac{\epsilon_{i}'}{p(\mu_i^+-\mu_i^-)}\leq\frac{2\epsilon_{i}'}{\eta\mu}\leq\frac{2\epsilon_{i}'}{\eta p\mu_{i}^{+}}.
\]
 Hence,
\[
\frac{p\mu_{i}^{+}}{\mu}\geq p^{2\epsilon_{i}'/(\eta p\mu_{i}^{+})}=\exp\left(-\frac{2\epsilon_{i}'\ln(1/p)}{\eta p\mu_{i}^{+}}\right).
\]
 Using the fact that $1-e^{-x}\leq x$ for all $x\geq0$, we have
\begin{align*}
\left(1-p\right)\frac{\mu_{i}^{-}}{\mu}=1-\frac{p\mu_{i}^{+}}{\mu}\leq1-\exp\left(-\frac{2\epsilon_{i}'\ln(1/p)}{\eta p\mu_{i}^{+}}\right)\leq\frac{2\epsilon_{i}'\ln(1/p)}{\eta p\mu_{i}^{+}}.
\end{align*}
 This implies
\begin{align*}
\mu_{i}^{-}\leq\left(\frac{\mu}{\eta p\mu_{i}^{+}}\right) \left(\frac{2}{1-p}\right) \epsilon_{i}'\ln(1/p)\leq\left(\frac{2}{\eta(2-\eta)}\right)\left(\frac{2}{\eta}\right)\epsilon_{i}'\ln(1/p)\leq C_{2}\epsilon_{i}'\ln(1/p),
\end{align*}
 provided we choose $C_{2}\geq\frac{4}{\eta^{2}(2-\eta)}$. We now
have
\[
pI_{i}[f]=p(\mu_{i}^{+}-\mu_{i}^{-})=\mu-\mu_{i}^{-}\geq\mu-C_{2}\epsilon_{i}'\ln(1/p)\geq\mu-C_{2}\epsilon\mu\ln(1/p),
\]
 so Case (2) occurs, as asserted.
\end{proof}

We now prove a version of Claim \ref{claim:2-cases-constant-p} for small $p$ and a general $f$ (i.e., not necessarily monotone increasing). Here, similarly to
in the monotone case, we obtain that either $\mu_{i}^{-}$ is small, or else
$pI_{i}[f]$ is small.

\begin{claim}
\label{claim:2-cases-small-p} There exists an absolute constant $C_2>0$ such that if $0 < p \leq e^{-2}$, then for each $i \in [n]$, one
of the following holds.
\begin{description}
\item [{Case (1)}] We have $pI_{i}[f]\leq C_2 \epsilon_{i}'\ln(1/p)$, and $\mu_{i}^{-}\geq \left(1-C_{2}\epsilon\ln(1/p)\right)\mu$.
\item [{Case (2)}] We have $\mu_{i}^{-}\le C_{2}\epsilon_{i}'\ln(1/p)$,
and $pI_{i}[f]\geq\left(1-C_{2}\epsilon\ln(1/p)\right)\mu$.
\end{description}
\end{claim}
\begin{proof}
By (\ref{eq:Inf_i is like monotone}), we have
\begin{equation}\label{eq:inf-ineq} pI_{i}[f]-p\left|\mu_{i}^{+}-\mu_{i}^{-}\right|\le\epsilon_{i}'.\end{equation}
Firstly, suppose that $\mu_i^- \geq \mu_i^+$; then $\mu_i^- \geq \mu$, so clearly we have $\mu_i^- \geq (1-C_2\epsilon \ln(1/p))\mu$ for any $C_2>0$. Moreover, by Lemma \ref{lemma:analysis-basic-functions} and (\ref{eq:rearranged-2}), we have
\begin{equation} \label{eq:diff-ineq} \left(\mu_{i}^{-}-\mu_{i}^{+}\right)\tfrac{p}{2}\le H\left(\mu_{i}^{+},\mu_{i}^{-}\right)-G\left(\mu_{i}^{+},\mu_{i}^{-}\right)\le\epsilon_{i}^{'}\end{equation}
Combining (\ref{eq:inf-ineq}) and (\ref{eq:diff-ineq}) yields $p I_i[f] \leq 3\epsilon_i'$, so Case (1) holds.

Secondly, suppose that $\mu_i^+ > \mu_i^-$. By Lemma \ref{lemma:analysis-basic-functions} and equation (\ref{eq:rearranged-1}), we have
\[
p(\mu_i^+-\mu_i^-) \log_{p}\left(\frac{p\mu_{i}^{+}}{\mu}\right)\le F\left(\mu_{i}^{+},\mu_{i}^{-}\right)-G\left(\mu_{i}^{+},\mu_{i}^{-}\right)\leq\epsilon_{i}'.
\]
Similarly to in the proof of Claim \ref{claim:mon-2-cases}, we now split into two cases.
\begin{caseenv}
\item[Case (a):] $\mu_{i}^{+}\leq \frac{\mu}{2p}$.
\end{caseenv}

If $\mu_{i}^{+}\leq\frac{\mu}{2p}$, then Case (1)
of Claim \ref{claim:mon-2-cases} must occur, provided we take $C_{2}$
to be sufficiently large. Indeed, we then have
\[
p(\mu_i^+-\mu_i^-) \log_{p}(1/2)\leq p(\mu_i^+-\mu_i^-)\log_{p}\left(\frac{p\mu_{i}^{+}}{\mu}\right)\leq\epsilon_{i}',
\]
 which, in combination with (\ref{eq:inf-ineq}), gives $pI_{i}[f] \leq\frac{1}{\ln 2}\epsilon_{i}'\ln(1/p) +\epsilon_i'\leq C_{2}\epsilon_{i}'\ln(1/p)$,
provided we choose $C_{2}\geq1/(\ln 2) + 1/2$. This in turn implies that
\[
\mu_{i}^{-}=\mu-p(\mu_{i}^{+}-\mu_{i}^{-})\geq \mu-pI_{i}[f]\geq\mu-C_{2}\epsilon_{i}'\ln(1/p)\geq\mu-C_{2}\epsilon\mu\ln(1/p),
\]
 so Case (1) occurs, as asserted.
\begin{caseenv}
\item[Case (b):] $\mu_{i}^{+}\ge \frac{\mu}{2p}$.
\end{caseenv}

If $\mu_{i}^{+} \geq \frac{\mu}{2p}$, then Case (2)
of Claim \ref{claim:mon-2-cases} must occur. Indeed, since $\mu_{i}^{-}\leq\mu$,
we have $p(\mu_{i}^{+}-\mu_{i}^{-})\geq\left(\tfrac{1}{2}-p\right)\mu\geq\tfrac{1}{3}\mu$.
We now have
\[
\log_{p}\left(\frac{p\mu_{i}^{+}}{\mu}\right)\le\frac{\epsilon_{i}'}{p(\mu_i^+-\mu_i^-)}\leq\frac{3\epsilon_{i}'}{\mu}\leq\frac{3\epsilon_{i}'}{p\mu_{i}^{+}}.
\]
 Hence,
\[
\frac{p\mu_{i}^{+}}{\mu}\geq p^{3\epsilon_{i}'/(p\mu_{i}^{+})}=\exp\left(-\frac{3\epsilon_{i}'\ln(1/p)}{p\mu_{i}^{+}}\right).
\]
 Using the fact that $1-e^{-x}\leq x$ for all $x\geq0$, we have
\begin{align*}
\left(1-p\right)\frac{\mu_{i}^{-}}{\mu}=1-\frac{p\mu_{i}^{+}}{\mu}\leq1-\exp\left(-\frac{3\epsilon_{i}'\ln(1/p)}{p\mu_{i}^{+}}\right)\leq\frac{3\epsilon_{i}'\ln(1/p)}{p\mu_{i}^{+}}.
\end{align*}
 This implies
$$\mu_{i}^{-}\leq\left(\frac{\mu}{p\mu_{i}^{+}}\right) \left(\frac{3}{1-p}\right) \epsilon_{i}'\ln(1/p)\leq \frac{6\epsilon_{i}'\ln(1/p)}{1-e^{-2}} \leq C_{2}\epsilon_{i}'\ln(1/p),$$
 provided we choose $C_{2}\geq\frac{6}{1-e^{-2}}$. We now
have
\[
pI_{i}[f]\geq p(\mu_{i}^{+}-\mu_{i}^{-})=\mu-\mu_{i}^{-}\geq\mu-C_{2}\epsilon_{i}'\ln(1/p)\geq\mu-C_{2}\epsilon\mu\ln(1/p),
\]
 so Case (2) occurs, as asserted.
\end{proof}

\subsection{Either $\mu$ is fairly small, or very close to 1}

Here, we show that there exists a constant $c_{4}>0$ such that either $\mu = 1-O\left(\frac{\epsilon \ln (1/p)}{\log\left(\frac{1}{\epsilon \ln (1/p)}\right)}\right)$ (i.e., $\mu$ is very close to 1),
or else $\mu < 1-c_{4}$ (i.e., $\mu$ is bounded away from 1). For a general $f$ (and $0 < p \leq 1/2)$, we obtain this by applying the $p$-biased isoperimetric inequality to the complement of $f$: $\tilde{f}\left(x\right)=1-f\left(x\right)$. For monotone $f$ (and $0 < p < 1$), we apply the $p$-biased isoperimetric inequality to the
dual of $f$: $f^{\ast}(x) =1-f\left(\overline{x}\right) = 1-f(1-x)$.
\begin{claim}
\label{claim:expectation-constraint-1}
Let $0 < p\leq1/2$.
Then we either have
\[
\mu\geq1-\frac{C_{3}\epsilon\ln(1/p)}{\ln\left(\frac{1}{\epsilon\ln(1/p)}\right)},
\]
 or else $\mu\leq 1-c_{4}$, where $C_{3},c_{4}>0$ are absolute constants with $c_4 < 1/4$.
 \end{claim}
\begin{proof}
Note that $\mu_{p}(\tilde{f})=1-\mu_{p}(f)$ and that $I^{p}[\tilde{f}]=I^{p}[f]$.
By assumption, we have $pI[f] = \mu(\log_{p}\mu+\epsilon)$. On the
other hand, applying Theorem \ref{thm:edge-iso-biased} to $\tilde{f}$, we obtain
\[
pI[f]=pI[\tilde{f}]\geq\left(1-\mu\right)\log_{p}\left(1-\mu\right).
\]
 Combining these two facts, we obtain
\[
\mu\left(\log_{p}\mu+\epsilon\right)\geq\left(1-\mu\right)\log_{p}\left(1-\mu\right).
\]
Suppose that $\delta:=1-\mu \le c_{4}$, where $c_4>0$ is to be chosen later. Then
\begin{align*}
\delta\log_{p}\left(\delta\right) & \le\left(1-\delta\right)\left(\log_{p}\left(1-\delta\right)+\epsilon\right)\\
 & =\left(1-\delta\right)\left(\frac{\ln\left(\frac{1}{1-\delta}\right)}{\ln\left(\frac{1}{p}\right)}+\epsilon\right)\le\frac{2\delta}{\ln\left(\frac{1}{p}\right)}+\epsilon,
\end{align*}
where the last inequality holds provided $c_{4}$ is sufficiently small.
Hence,
\[
\delta\left(\ln\left(\frac{1}{\delta}\right)-2\right)\le\epsilon\ln\left(\frac{1}{p}\right).
\]
Provided $c_4$ is sufficiently small, this implies that
$$\delta = O \left(\frac{\epsilon \ln(1/p)}{\ln\left(\frac{1}{\epsilon \ln(1/p)}\right)}\right),$$
proving the claim.
\end{proof}

\begin{claim}
\label{claim:expectation-constraint} For any $\eta >0$, there exist $C_3 = C_3(\eta)$ and $c_4 = c_4(\eta)>0$ such that the following holds. Suppose that $0 < p\leq1-\eta$, and suppose that $f$ is monotone increasing. Then we either have
\[
\mu\geq1-\frac{C_{3}\epsilon}{\ln\left(\frac{1}{\epsilon}\right)},
\]
 or else $\mu\leq 1-c_{4}$.
 \end{claim}
\begin{proof}
Note that $f^{\ast}$ is monotone increasing, since $f$ is. Moreover, $\mu_{1-p}(f^{\ast})=1-\mu_{p}(f)$ and $I^{1-p}[f^{\ast}]=I^{p}[f]$. By assumption, we have $pI^p[f] = \mu(\log_{p}\mu+\epsilon)$. On the
other hand, applying Theorem \ref{thm:edge-iso-biased} to $f^{\ast}$,
we obtain
\[
(1-p)I^{p}[f]=(1-p)I^{1-p}[f^{\ast}]\geq\left(1-\mu\right)\log_{1-p}\left(1-\mu\right).
\]
 Combining these two facts, we obtain
\begin{align*}
\mu\left(\log_{p}\mu+\epsilon\right) & \geq\frac{p}{1-p}\left(1-\mu\right)\log_{1-p}\left(1-\mu\right).
\end{align*}
Suppose that $\delta:=1-\mu\leq c_{4}$, where $c_4 = c_4(\eta)>0$ is to be chosen later. Then
\begin{align}
\frac{p}{1-p}\delta\log_{1-p}\left(\delta\right) & \le\left(1-\delta\right)\left(\log_{p}\left(1-\delta\right)+\epsilon\right)\nonumber \\
 & =\left(1-\delta\right)\left(\frac{\ln\left(\frac{1}{1-\delta}\right)}{\ln\left(\frac{1}{p}\right)}+\epsilon\right)\le\frac{2\delta}{\ln\left(\frac{1}{p}\right)}+\epsilon,\label{eq:measure can't be too much close to 1}
\end{align}
where the last inequality holds provided $c_{4}$ is sufficiently small. Observe that $\ln\left(\frac{1}{1-p}\right)=\Theta_{\eta}\left(p\right)$.
Hence,
\begin{equation}
\frac{p}{1-p}\delta\log_{1-p}\left(\delta\right)=\frac{p}{1-p}\delta\frac{\ln\left(\frac{1}{\delta}\right)}{\ln\left(\frac{1}{1-p}\right)}=
\Theta_{\eta}\left(\frac{p}{1-p}\delta
\frac{\ln\left(\frac{1}{\delta}\right)}{p}\right)=\Theta_{\eta}\left(\delta\ln\left(\frac{1}{\delta}\right)\right).
\label{eq:measure can't be too much close to 2}
\end{equation}
Combining (\ref{eq:measure can't be too much close to 1})
and (\ref{eq:measure can't be too much close to 2}), we obtain
\[
\Theta_{\eta} (\delta\ln(1/\delta)) -\frac{2\delta}{\eta} \leq \Theta_{\eta} (\delta\ln(1/\delta)) -\frac{2\delta}{\ln\left(1/(1-\eta)\right)} \leq \Theta_{\eta} (\delta\ln(1/\delta)) -\frac{2\delta}{\ln\left(1/p\right)}\leq \epsilon,
\]
using the fact that $1-\eta \leq e^{-\eta}$ for all $\eta \in \mathbb{R}$. This in turn implies that
$$\delta = O_{\eta}\left(\frac{\epsilon}{\ln\left(\frac{1}{\epsilon}\right)}\right)$$
provided $c_4$ is sufficiently small depending on $\eta$, proving the claim.
\end{proof}

\subsection{There exists an influential coordinate}

We now show that unless $\mu$ is very close to $1$, there
must exist a coordinate whose influence is large. This coordinate
will be used in the inductive step of the proof of our two stability theorems. First, we deal with the case of small $p$ and general $f$ (i.e., $f$ not necessarily monotone increasing).
\begin{claim}
\label{claim:influential-coordinate} For any $\zeta_0 \in (0,c_4/2)$, the following holds provided $c_0$ is sufficiently small (depending on the absolute constants $C_2,C_3$ and $c_4$). Suppose that $0 < p < \zeta_0$ and $\epsilon \ln(1/p) \leq c_0$. If $\mu<1-\frac{C_{3}\epsilon\ln(1/p)}{\ln\left(\frac{1}{\epsilon\ln(1/p)}\right)}$,
then there exists $i\in[n]$ for which Case (2) of Claim \ref{claim:2-cases-small-p} occurs, i.e. $\mu_i^- \leq C_2 \epsilon_i' \ln(1/p)$ and $p I_i[f] \geq (1-C_2 \epsilon \ln(1/p))\mu$.
 (Here, $C_2$ is the absolute constant from Claim \ref{claim:2-cases-small-p}, and $C_3,c_4$ are the absolute constants from Claim \ref{claim:expectation-constraint-1}.)
 \end{claim}
\begin{proof}
We prove the claim by induction on $n$.

If $n=1$ and $\mu<1-\frac{C_{3}\epsilon\ln(1/p)}{\ln\left(\frac{1}{\epsilon\ln(1/p)}\right)}$,
then by Claim \ref{claim:expectation-constraint-1}, we have $\mu<1-c_4$, and therefore $f\equiv0$ or $f=1_{\{x_{1}=1\}}$. (If $f = 1_{\{x_1 = 0\}}$ then $\mu = 1-p > 1-\zeta_0 > 1-c_4/2$.) Hence, we have $\mu_{1}^{-}=0$, so Case (2) must occur for the coordinate 1, verifying the base case.

We now do the inductive step. Let $n\geq2$, and assume the claim
holds when $n$ is replaced by $n-1$. Let $f$ be as in the statement of the claim; then by Claim \ref{claim:expectation-constraint-1}, we have $\mu \leq 1-c_4$. Suppose for a contradiction
that $f$ has Case (1) of Claim \ref{claim:2-cases-small-p} occurring for each $i\in[n]$. First, suppose that $\epsilon_{i}^{-}\geq\epsilon$
for each $i\in[n]$. Fix any $i\in[n]$. By (\ref{eq:rearranged-1}),
we have $0\leq\epsilon_{i}'\leq p(\epsilon-\epsilon_{i}^{+})\mu_{i}^{+}$,
so $\epsilon_{i}^{+}\leq\epsilon$ and therefore
\begin{equation}
\label{eq:influence-bound-1}
I_{i}[f]\leq \tfrac{1}{p} C_{2}\epsilon_{i}'\ln(1/p)\leq C_{2}\left(\epsilon-\epsilon_{i}^{+}\right)\mu_{i}^{+}\ln(1/p)\leq C_{2}c_{0}\mu_{i}^{+},
\end{equation}
using our assumption that $\epsilon \ln(1/p) \leq c_0$. Hence,
\[
\mu_{i}^{+}-\mu\leq|\mu_{i}^{+}-\mu_{i}^{-}|\le I_{i}[f]\leq C_{2}c_{0}\mu_{i}^{+},
\]
 so
\[
\mu_{i}^{+}\leq\frac{\mu}{1-C_{2}c_{0}} \leq\frac{1-c_{4}}{1-C_{2}c_{0}}<1-\frac{C_3 c_0}{\ln(1/c_0)}\leq 1-\frac{C_{3}\epsilon\ln(1/p)}{\ln\left(\frac{1}{\epsilon\ln(1/p)}\right)}\leq1-\frac{C_{3}\epsilon_{i}^{+}\ln(1/p)}{\ln\left(\frac{1}{\epsilon_{i}^{+}\ln(1/p)}\right)},
\]
 provided $c_{0}$ is sufficiently small (depending on $C_2$, $C_3$ and $c_4$). It follows that $f_{i\to1}$
satisfies the hypothesis of the claim, for each $i\in[n]$. Hence,
by the induction hypothesis, there exists $j\in[n]\setminus\{i\}$
such that $f_{i\to1}$ has Case (2) of Claim \ref{claim:2-cases-small-p}
occurring for the coordinate $j$, so
\[
pI_{j}[f_{i\to1}]\geq\left(1-C_{2}\epsilon_{i}^{+}\ln(1/p)\right)\mu_{i}^{+}.
\]
 We now have
\begin{align*}
I_{j}[f]\geq pI_{j}[f_{i\to1}]\geq\left(1-C_{2}\epsilon_{i}^{+}\ln(1/p)\right)\mu_{i}^{+}\geq(1-C_{2}\epsilon\ln(1/p))\mu_{i}^{+}\geq(1-C_{2}c_{0})\mu_{i}^{+},
\end{align*}
 but this contradicts the fact that (\ref{eq:influence-bound-1})
holds when $i$ is replaced by $j$, provided $c_{0}$ is sufficiently
small (depending on $C_2$).

We may assume henceforth that there exists $i\in[n]$ such that $\epsilon_{i}^{-}<\epsilon$.
Fix such a coordinate $i$. Since Case (1) occurs for the coordinate
$i$, we have
\begin{equation}\label{eq:mu-lower}
\mu_{i}^{-}\geq(1-C_{2}\epsilon\ln(1/p))\mu\geq(1-C_{2}c_{0})\mu.
\end{equation}
 On the other hand, we have
\begin{align*}
\mu_{i}^{-}& \le\frac{\mu}{1-p} \le\frac{1-c_4}{1-\zeta_0} < \frac{1-c_4}{1-c_4/2} \leq 1-\frac{C_3 c_0}{\ln(1/c_0)}\\
&\leq 1-\frac{C_{3}\epsilon\ln(1/p)}{\ln\left(\frac{1}{\epsilon\ln(1/p)}\right)} <1-\frac{C_{3}\epsilon_{i}^{-}\ln(1/p)}{\ln\left(\frac{1}{\epsilon_{i}^{-}\ln(1/p)}\right)},
\end{align*}
provided $c_0$ is sufficiently small (depending on $C_3$ and $c_4$). Hence, $f_{i\to0}$ satisfies
the hypotheses of the claim. Therefore, by the induction hypothesis, there
exists $j\in[n]\setminus\{i\}$ such that $f_{i\to0}$ has Case (2)
of Claim \ref{claim:2-cases-small-p} occurring for the coordinate
$j$, so
\[
pI_{j}[f_{i\to0}]\geq\left(1-C_{2}\epsilon_{i}^{-}\ln(1/p)\right)\mu_{i}^{-}.
\]
 Therefore, we have
\begin{align*}
pI_{j}[f] & \geq p(1-p)I_{j}[f_{i\to0}]\geq(1-p)\left(1-C_{2}\epsilon_{i}^{-}\ln(1/p)\right)\mu_{i}^{-}\\
 & >(1-p)\left(1-C_{2}\epsilon\ln(1/p)\right)(1-C_{2}c_{0})\mu\geq \tfrac{1}{2}(1-C_{2}c_{0})^{2} \mu,
\end{align*}
using (\ref{eq:mu-lower}) for the third inequality, contradicting the fact that $f$ satisfies Case (1) of Claim \ref{claim:2-cases-small-p} for the coordinate
$j$, provided $c_0$ is sufficiently small (depending on $C_2$). This completes the inductive step, proving the claim.
\end{proof}

Now we deal with the case of $p$ bounded away from $0$ and bounded from above by $1/2$, and arbitrary $f$.
\begin{claim}
\label{claim:influential-coordinate-2}  For each $\zeta >0$, the following holds provided $c_0$ is sufficiently small depending on $\zeta$. Suppose that $\zeta < p \leq 1/2$ and $\epsilon \ln(1/p) \leq c_0$. If $\mu<1-\frac{C_{3}\epsilon\ln(1/p)}{\ln\left(\frac{1}{\epsilon\ln(1/p)}\right)}$,
then there exists $i\in[n]$ for which Case (2) of Claim \ref{claim:2-cases-constant-p} occurs, i.e. $\min\{\mu_i^-,\mu_i^+\} \leq C_2 \epsilon_i'$ and $I_i[f] \geq (1-C_2 \epsilon)\mu$.
 (Here, $C_2 =C_2(\zeta)$ is the constant from Claim \ref{claim:2-cases-constant-p}, and $C_3$ is the absolute constant from Claim \ref{claim:expectation-constraint-1}.)
 \end{claim}

\begin{proof}
If $n=1$ and $\mu<1-\frac{C_{3}\epsilon\ln(1/p)}{\ln\left(\frac{1}{\epsilon\ln(1/p)}\right)}$,
then we must have either $f\equiv0$, $f=1_{\{x_{1}=1\}}$ or $f =1_{\{x_1 = 0\}}$. Hence, we have $\min\{\mu_i^+,\mu_{1}^{-}\}=0$, so Case (2) of Claim \ref{claim:2-cases-constant-p} must occur for the coordinate 1, verifying the base case.

We now do the inductive step. Let $n\geq2$, and assume the claim holds when $n$ is replaced by $n-1$. Let $f$ be as in the statement of the claim; then by Claim \ref{claim:expectation-constraint-1}, we have $\mu \leq 1-c_4$. Suppose for a contradiction
that $f$ has Case (1) of Claim \ref{claim:2-cases-constant-p} occurring for each $i\in[n]$. First, suppose that $\epsilon_{i}^{-}\geq\epsilon$
for each $i\in[n]$. Then almost exactly the same argument as in the proof of Claim \ref{claim:influential-coordinate} yields a contradiction, provided $c_0$ is sufficiently small depending on $\zeta$. Therefore, we may assume henceforth that there exists $i \in [n]$ such that $\epsilon_i^- < \epsilon$. By assumption, Case 1 of Claim \ref{claim:2-cases-constant-p} occurs for the coordinate $i$, and therefore $\min\{\mu_i^+,\mu_i^-\} \geq (1-C_2 \epsilon)\mu$. It follows that
\begin{align*} \mu_i^- & = \frac{\mu - p\mu_i^+}{1-p} \leq \frac{1 - p(1-C_2\epsilon)}{1-p}\mu = \mu + \frac{pC_2\epsilon \mu}{1-p} \leq 1-c_4 +  C_2\epsilon < 1-\frac{C_3c_0}{\ln(1/c_0)}\\
& \leq 1-\frac{C_{3}\epsilon\ln(1/p)}{\ln\left(\frac{1}{\epsilon\ln(1/p)}\right)} <1-\frac{C_{3}\epsilon_{i}^{-}\ln(1/p)}{\ln\left(\frac{1}{\epsilon_{i}^{-}\ln(1/p)}\right)},
\end{align*}
provided $c_0$ is sufficiently small depending on $\zeta$. Hence, $f_{i\to0}$ satisfies
the hypotheses of the claim. Therefore, by the induction hypothesis, there
exists $j\in[n]\setminus\{i\}$ such that $f_{i\to0}$ has Case (2)
of Claim \ref{claim:2-cases-small-p} occurring for the coordinate
$j$, so
\[
I_{j}[f_{i\to0}]\geq\left(1-C_{2}\epsilon_{i}^{-} \right)\mu_{i}^{-}.
\]
 We now have
\begin{align*}
I_{j}[f] & \geq (1-p)I_{j}[f_{i\to0}]\geq (1-p)\left(1-C_{2}\epsilon_{i}^{-}\right)\mu_{i}^{-}\\
 & >(1-p)\left(1-C_{2}\epsilon \right)^2 \mu\geq \tfrac{1}{2} (1-C_{2}c_{0} / \ln(2) )^{2} \mu,
\end{align*}
 contradicting the fact that $f$ satisfies Case (1) of Claim \ref{claim:2-cases-constant-p} for the coordinate
$j$, provided $c_0$ is sufficiently small depending on $C_2$ (i.e., on $\zeta$). This completes the inductive step, proving the claim.
\end{proof}

Finally, we deal with the case of monotone $f$ and all $p$ bounded away from 1.

\begin{claim}
\label{claim:influential-coordinate-3}
For each $\eta>0$, the following holds provided $c_0$ is sufficiently small depending on $\eta$. Let $0 < p \leq 1-\eta$, and suppose $f$ is monotone increasing. If $\mu < 1-\frac{C_{3}\epsilon\ln(1/p)}{\ln\left(\frac{1}{\epsilon\ln(1/p)}\right)}$, then there exists $i \in [n]$ for which Case (2) of Claim \ref{claim:mon-2-cases} occurs. (Here, $C_3 = C_3(\eta)$ is the constant from Claim \ref{claim:expectation-constraint}.)
\end{claim}

\begin{proof}
We prove the claim by induction on $n$.

If $n=1$ and $\mu < 1-\frac{C_{3}\epsilon\ln(1/p)}{\ln\left(\frac{1}{\epsilon\ln(1/p)}\right)}$, then since $\mu <1$ we must have either $f \equiv 0$ or $f = 1_{\{x_1 = 1\}}$, so $\mu_{1}^{-}=0$. Hence, Case (2) of Claim \ref{claim:mon-2-cases} occurs for the coordinate 1, verifying the base case.

We now do the inductive step. Let $n \geq 2$, and assume the claim holds when $n$ is replaced by $n-1$. Let $f$ be as in the statement of the claim; then by Claim \ref{claim:expectation-constraint}, we have $\mu \leq 1-c_4$. Suppose for a contradiction that $f$ has Case (1) of Claim \ref{claim:mon-2-cases} occurring for each $i \in [n]$. First, suppose that $\epsilon_{i}^{-} \geq \epsilon$ for each $i \in [n]$. Fix any $i \in [n]$. Then almost exactly the same argument as in the proof of Claim \ref{claim:influential-coordinate} (using Claim \ref{claim:expectation-constraint} in place of Claim \ref{claim:expectation-constraint-1}) yields a contradiction.

We may therefore assume henceforth that there exists $i \in [n]$ such that $\epsilon_i^- < \epsilon$. Since Case (1) of Claim \ref{claim:mon-2-cases} occurs for the coordinate $i$, we have
$$\mu_i^- \geq (1-C_2 \epsilon \ln (1/p))\mu \geq (1-C_2c_0) \mu.$$
On the other hand, we have
$$ \mu_i^- \leq \mu < 1-\frac{C_{3}\epsilon\ln(1/p)}{\ln\left(\frac{1}{\epsilon\ln(1/p)}\right)} < 1-\frac{C_{3}\epsilon_i^-\ln(1/p)}{\ln\left(\frac{1}{\epsilon_i^-\ln(1/p)}\right)},$$
so $f_{i \to 0}$ satisfies the hypotheses of the claim. Hence, by the induction hypothesis, there exists $j \in [n] \setminus \{i\}$ such that $f_{i \to 0}$ has Case (2) of Claim \ref{claim:mon-2-cases} occurring for the coordinate $j$, so
$$p I_{j}[f_{i\to0}] \geq \left(1-C_{2}\epsilon_i^- \ln(1/p)\right) \mu_{i}^{-}.$$
We now have
\begin{align*} pI_j[f] & \geq p(1-p)I_j[f_{i \to 0}] \geq (1-p) \left(1-C_{2}\epsilon_i^- \ln(1/p)\right) \mu_{i}^{-}\\
& > (1-p) \left(1-C_{2}\epsilon \ln(1/p)\right) (1-C_2c_0)\mu \geq \eta (1-C_2 c_0)^2\mu,
\end{align*}
contradicting the fact that $f$ satisfies Case (1) of Claim \ref{claim:mon-2-cases} for the coordinate $j$, provided $c_0$ is sufficiently small depending on $\eta$. This completes the inductive step, proving the claim.
\end{proof}

\subsection{Bootstrapping}

Our final required ingredient is a `bootstrapping' argument, which
says that if $\min\left\{ \mu_{i}^{-},\mu_{i}^{+}\right\} $ is `somewhat'
small, then it must be `very' small.
\begin{claim}
\label{claim:bootstrapping-1} Let $\zeta\in (0,1/2)$. There exist $C_{5}=C_{5}(\zeta)>0$ and $c_{5} = c_5(\zeta) >0$ such that the following holds. Let $\zeta<p\le\frac{1}{2}$. If $\mu_{i}^{-} \le c_{5}\mu$,
then
\[
\mu_i^- \le\frac{C_{5}\left(\epsilon-\epsilon_{i}^{+}\right)\ln(1/p)}{\ln\left(1/\left((\epsilon-\epsilon_{i}^{+})\ln(1/p)\right)\right)}\mu,
\]
and if
 $\mu_{i}^{+} \le c_{5}\mu$, then
 \[
\mu_i^+\le\frac{C_{5}\left(\epsilon-\epsilon_{i}^{-}\right)\ln(1/p)}{\ln\left(1/\left((\epsilon-\epsilon_{i}^{-})\ln(1/p)\right)\right)}\mu.
\]
 \end{claim}
\begin{proof}
Let $c_5 = c_5(\zeta)>0$ to be chosen later. First suppose that $\mu_{i}^{-}\le\mu_{i}^{+}$, and write $\delta : = \mu_i^- / \mu$; then $\delta \leq c_5$. Using (\ref{eq:rearranged-1}), we have
\begin{align}
\label{eq:upper}
\left(1-p\right)\mu_{i}^{-}\log_{p}\left(\mu_{i}^{-}\right)+ & p\mu_{i}^{+}\log_{p}\mu_{i}^{+}-\mu\log_{p}\mu+pI_{i}\left[f\right]\nonumber \\
 & =\left(\epsilon-\epsilon_i^{+}\right)p\mu_{i}^{+}+(\epsilon-\epsilon_{i}^{-})\left(1-p\right)\mu_{i}^{-} \nonumber \\
 & \leq \left(\epsilon-\epsilon_i^{+}\right)\mu +\epsilon\left(1-p\right)\mu_{i}^{-},
\end{align}
the last inequality following from the fact that $p\mu_i^+ = \mu-(1-p)\mu_i^- \leq \mu$.
Observe that
\begin{align}
\label{eq:lower}
\textrm{LHS} & =\left(1-p\right)\mu_{i}^{-}\log_{p}\left(\mu_{i}^{-}\right)+p\mu_{i}^{+}\log_{p}\left(p\mu_{i}^{+}\right)-\mu\log_{p}\mu+pI_{i}\left[f\right]-p\mu_{i}^{+} \nonumber \\
 & \ge\left(1-p\right)\mu_{i}^{-}\log_{p}\left(\mu_{i}^{-}\right)+p\mu_{i}^{+}\log_{p}\left(\mu\right)-\mu\log_{p}\left(\mu\right)-p\mu_{i}^{-} \nonumber \\
 & =\left(1-p\right)\mu_{i}^{-}\log_{p}\left(\frac{\mu_{i}^{-}}{\mu}\right)-p\mu_{i}^{-}.
\end{align}
Combining (\ref{eq:upper}) and (\ref{eq:lower}) and rearranging, we obtain
\begin{equation}
\left(\frac{\mu_{i}^{-}}{\mu}\right)\left(\log_{p}\left(\frac{\mu_{i}^{-}}{\mu}\right)-\frac{p}{1-p} -\epsilon\right)\leq\frac{\epsilon-\epsilon_{i}^{+}}{1-p} \leq 2(\epsilon-\epsilon_{i}^+).\label{eq:bootstrap-1}
\end{equation}
It follows that
$$\delta \left(\ln (1/\delta)-\frac{p\ln(1/p)}{1-p} - \epsilon \ln(1/p)\right)\leq 2(\epsilon-\epsilon_{i}^+) \ln(1/p),$$
and therefore
$$\delta \left(\ln (1/\delta)-\frac{2}{e} - c_0\right)\leq 2(\epsilon-\epsilon_{i}^+) \ln(1/p),$$
using the fact that $p\ln(1/p) / (1-p) \leq 2/e$ whenever $0 \leq p \leq 1/2$.
Since $\delta \leq c_5$, if $c_{5}$ is sufficiently small this clearly implies that
$$\delta = O_{\zeta}\left(\frac{(\epsilon-\epsilon_i^+)\ln(1/p)}{\ln\left(\frac{1}{(\epsilon-\epsilon_i^+)\ln(1/p)}\right)}\right),$$
as required.

Now suppose that $\mu_{i}^{+}\le\mu_{i}^{-}$, and write $\delta : = \mu_i^+ / \mu$; then $\delta \leq c_5$. Using (\ref{eq:rearranged-1}), we have
\begin{align}
\label{eq:upper2}
\left(1-p\right)\mu_{i}^{-}\log_{p}\left(\mu_{i}^{-}\right)+ & p\mu_{i}^{+}\log_{p}\mu_{i}^{+}-\mu\log_{p}\mu+pI_{i}\left[f\right]\nonumber \\
 & =\left(\epsilon-\epsilon_i^{+}\right)p\mu_{i}^{+}+(\epsilon-\epsilon_{i}^{-})\left(1-p\right)\mu_{i}^{-} \nonumber \\
 & \leq \epsilon p\mu_i^+ +(\epsilon-\epsilon_i^-)\mu.
\end{align}
Observe that
\begin{align}
\label{eq:lower2}
\textrm{LHS} & =\left(1-p\right)\mu_{i}^{-}\log_{p}\left((1-p)\mu_{i}^{-}\right)+p\mu_{i}^{+}\log_{p}\left(\mu_{i}^{+}\right)-\mu\log_{p}(\mu)+pI_{i}\left[f\right] \nonumber \\
& -(1-p) \mu_{i}^{-} \log_p(1-p) \nonumber \\
 & \geq \left(1-p\right)\mu_{i}^{-}\log_{p} (\mu) +p\mu_{i}^{+}\log_{p}\left(\mu_{i}^{+}\right)-\mu\log_{p}(\mu)+p(\mu_i^- - \mu_i^+) \nonumber\\
& -(1-p) \mu_{i}^{-} \log_p(1-p)\nonumber\\
& = p\mu_i^+ \log_p\left(\frac{\mu_{i}^{+}}{\mu}\right)+ \mu_i^-(p - (1-p)\log_p(1-p)) - p\mu_i^+ \nonumber\\
& = p\mu_i^+ \log_p\left(\frac{\mu_{i}^{+}}{\mu}\right)+ K(p) \mu_i^- - p\mu_i^+ \nonumber\\
& \geq p\mu_i^+ \log_p\left(\frac{\mu_{i}^{+}}{\mu}\right) - p\mu_i^+.
\end{align}
Combining (\ref{eq:upper2}) and (\ref{eq:lower2}) and rearranging, we obtain
\begin{equation}
\left(\frac{\mu_{i}^{+}}{\mu}\right)\left(\log_{p}\left(\frac{\mu_{i}^{+}}{\mu}\right)-1-\epsilon\right)\leq\frac{\epsilon-\epsilon_{i}^{-}}{p} \leq \tfrac{1}{\zeta}(\epsilon-\epsilon_{i}^-).\label{eq:bootstrap-2}
\end{equation}
It follows that
$$\delta \left(\ln (1/\delta)-\ln(1/p) - \epsilon \ln(1/p)\right)\leq \tfrac{1}{\zeta} (\epsilon-\epsilon_{i}^-) \ln(1/p),$$
and therefore
$$\delta \left(\ln (1/\delta)-\ln(1/\zeta) - c_0\right)\leq \tfrac{1}{\zeta}(\epsilon-\epsilon_{i}^-) \ln(1/p).$$
Since $\delta \leq c_5$, if $c_{5}$ is sufficiently small (depending on $\zeta$), this clearly implies that
$$\delta = O_{\zeta}\left(\frac{(\epsilon-\epsilon_i^-)\ln(1/p)}{\ln\left(\frac{1}{(\epsilon-\epsilon_i^-)\ln(1/p)}\right)}\right),$$
as required.
\end{proof}
We now prove a bootstrapping claim suitable for use in the cases where $p \leq \zeta$ and $f$ is arbitrary, or where $p \leq 1-\eta$ and $f$ is monotone increasing.
\begin{claim}
\label{claim:bootstrapping-mon} Let $\eta >0$. There exist $C_{5}=C_{5}(\eta)>0$ and $c_{5} = c_5(\eta)>0$ such that if $p \leq 1-\eta$ and $\mu_{i}^{-}\le c_{5}\mu$, then
\[
\mu_{i}^{-}\le\frac{C_{5}\left(\epsilon-\epsilon_{i}^{+}\right)\ln(1/p)}{\ln\left(1/\left(\left(\epsilon-\epsilon_{i}^{+}\right)\ln(1/p)\right)\right)}\mu.
\]
 \end{claim}
\begin{proof}
As in the proof of Claim \ref{claim:bootstrapping-1}, we have
\begin{equation}
\left(\frac{\mu_{i}^{-}}{\mu}\right)\left(\log_{p}\left(\frac{\mu_{i}^{-}}{\mu}\right)-\frac{p}{1-p} - \epsilon\right)\leq\frac{\epsilon-\epsilon_{i}^{+}}{1-p}\leq\frac{\epsilon-\epsilon_{i}^{+}}{\eta}.\label{eq:bootstrap}
\end{equation}
 Writing $\delta:=\frac{\mu_{i}^{-}}{\mu}\le c_{5}$, we obtain
\[
\delta\left(\ln(1/\delta)-\frac{1}{e\eta} - c_0\right) \leq \delta\left(\ln(1/\delta)-\frac{p\ln(1/p)}{1-p} -\epsilon\ln(1/p)\right) \le\ln(1/p)O_{\eta}\left(\epsilon-\epsilon_{i}^{+}\right),
\]
using the fact that $p\ln(1/p) /(1-p) \leq 1/(e\eta)$ whenever $0 \leq p \leq 1-\eta$. Provided $c_{5} = c_5(\eta)>0$ is sufficiently small, this implies that
\[
\delta = O_{\eta}\left(\frac{\left(\epsilon-\epsilon_{i}^{+}\right)\ln(1/p)}{\ln\left(1/\left(\left(\epsilon-\epsilon_{i}^{+}\right)\ln(1/p)\right)\right)}\right),
\]
 as required.
\end{proof}

\subsection{Inductive proofs of Theorems \ref{thm:skewed-iso-stability} and \ref{thm:mon-iso-stability}}
\label{subsec:ind}
\begin{proof}[Proof of Theorem \ref{thm:skewed-iso-stability}.]
First, we choose any $\zeta_0 \in (0,c_4/2)$ (where $c_4$ is the absolute constant from Claim \ref{claim:expectation-constraint-1}), and we deal with the case of $p < \zeta_0$, using Claim \ref{claim:influential-coordinate}. In this case, we prove that the conclusion of Theorem \ref{thm:skewed-iso-stability} holds with $S$ a monotone increasing subcube.

We proceed by induction on $n$. If $n=1$, then $f$ is the indicator function of a monotone increasing subcube unless $f = 1_{\{x_1=0\}}$, so we may assume that $f = 1_{\{x_1=0\}}$. Then $\mu_p(f) = 1-p > 1-\zeta_0 > 1-c_4$, so by Claim \ref{claim:expectation-constraint-1}, we have
\[
\mu_{p}(f)\geq1-\frac{C_{3}\epsilon\ln(1/p)}{\ln\left(\frac{1}{\epsilon\ln(1/p)}\right)},
\]
 so the conclusion of the theorem holds with $S=\{0,1\}$.

We now do the inductive step. Let $n\geq2$, and assume that Theorem \ref{thm:skewed-iso-stability}
holds when $n$ is replaced by $n-1$. Let $f:\{0,1\}^n \to \{0,1\}$ satisfy the hypotheses of Theorem \ref{thm:skewed-iso-stability}. We may assume throughout that $\mu_{p}(f)\leq 1-c_4$, otherwise by Claim \ref{claim:expectation-constraint-1}, we have
\[
\mu_{p}(f)\geq1-\frac{C_{3}\epsilon\ln(1/p)}{\ln\left(\frac{1}{\epsilon\ln(1/p)}\right)},
\]
 so the conclusion of the theorem holds with $S=\{0,1\}^{n}$. Since $\mu_p(f) \leq 1-c_4$, by Claim \ref{claim:influential-coordinate}, there exists $i\in[n]$ such that $\mu_{i}^{-}\leq C_{2} \epsilon\mu\ln(1/p)$, so if $c_{0}$
is a sufficiently small absolute constant, we have $\mu_{i}^{-}\leq c_{5} \mu$, where $c_5$ is the absolute constant we obtain by applying Claim \ref{claim:bootstrapping-mon} with $\eta = 1-\zeta_0$. Hence, $\mu_{i}^{-}$ satisfies the hypothesis of Claim \ref{claim:bootstrapping-mon}. Therefore, we have
\begin{equation}\label{eq:stronger-upper}
\mu_i^- \le\frac{C_{5}\left(\epsilon-\epsilon_{i}^{+}\right)\ln(1/p)}{\ln\left(1/\left((\epsilon-\epsilon_{i}^{+})\ln(1/p)\right)\right)}\mu,
\end{equation}
where $C_5$ is the absolute constant we obtain by applying Claim \ref{claim:bootstrapping-mon} with $\eta = 1-\zeta_0$.
In particular, we have $\epsilon_i^+ \leq \epsilon$. By applying the induction
hypothesis to $f_{i\to1}$, we obtain
\[
\mu_{p}\left(f_{i\to1}\Delta1_{S_{T}}\right)\leq\frac{C_{1}\epsilon_{i}^{+}\ln(1/p)\mu_{i}^{+}}{\ln\left(\frac{1}{\epsilon_{i}^{+}\ln(1/p)}\right)}
\]
for some monotone increasing subcube $S_{T}=\{x\in\{0,1\}^{[n]\setminus\{i\}}:\ x_{j}=1\ \forall j\in T\}$, where $T \subset [n]$. Therefore, writing
\[
S_{T\cup\{i\}}:=\{x\in\{0,1\}^{n}:\ x_{j}=1\ \forall j\in T\cup\{i\}\},
\]
 we have
\begin{align*}
\mu_{p}\left(f\Delta1_{S_{T\cup\left\{ i\right\} }}\right) & \leq\left(1-p\right)\mu_{i}^{-}+p\mu_{p}\left(f_{i\to1}\Delta1_{S_{T}}\right)\\
 & \leq\left(1-p\right)\left(\frac{C_{5}(\epsilon-\epsilon_{i}^{+})\ln(1/p)\mu}{\ln\left(\frac{1}{\left[\epsilon-\epsilon_{i}^{+}\right]\ln(1/p)}\right)}\right)+\frac{C_{1}\epsilon_{i}^{+}\ln(1/p)p\mu_{i}^{+}}{\ln\left(\frac{1}{\epsilon_{i}^{+}\ln(1/p)}\right)}\\
 & \leq\frac{\left(C_{5}\left(\epsilon-\epsilon_{i}^{+}\right)+C_{1}\epsilon_{i}^{+}\right)\mu\ln(1/p)}{\ln\left(\frac{1}{\epsilon\ln(1/p)}\right)}\leq\frac{C_{1}\epsilon\mu\ln(1/p)}{\ln\left(\frac{1}{\epsilon\ln(1/p)}\right)},
\end{align*}
 provided $C_{1}\geq C_5$, using (\ref{eq:stronger-upper}). Hence, the conclusion of the theorem holds with $S=S_{T\cup\{i\}}$. This completes the inductive step, proving the theorem in the case $p < \zeta_0$.

Now we prove the theorem in the case $\zeta_0 \leq p \leq 1/2$.

We proceed again by induction on $n$. If $n=1$, then as before the theorem holds trivially. Let $n\geq2$, and assume Theorem \ref{thm:skewed-iso-stability} holds
when $n$ is replaced by $n-1$. Let $f:\{0,1\}^n \to \{0,1\}$ satisfy the hypotheses of Theorem \ref{thm:skewed-iso-stability}. As before, we may assume throughout that $\mu_{p}(f)\leq 1-c_4$, otherwise by Claim \ref{claim:expectation-constraint-1}, we have
\[
\mu_{p}(f)\geq1-\frac{C_{3}\epsilon\ln(1/p)}{\ln\left(\frac{1}{\epsilon\ln(1/p)}\right)},
\]
 so the conclusion of the theorem holds with $S=\{0,1\}^{n}$. Since $\mu_p(f) \leq 1-c_4$, by Claim \ref{claim:influential-coordinate-2} (applied with $\zeta=\zeta_0$), provided $c_{0}$ is sufficiently small depending on $\zeta_0$, there exists $i\in[n]$ such that $\min\{\mu_{i}^{-},\mu_i^+\}\leq C_{2}(\zeta_0) \epsilon_i' \leq C_2(\zeta_0) \epsilon\mu$, so we have
$$\min\{\mu_{i}^{-},\mu_i^+\}\leq C_2(\zeta_0) \epsilon \mu \leq C_2(\zeta_0) c_0 /\ln(1/p) \leq C_2(\eta_0) c_0 /\ln(2) \leq c_{5}(\zeta_0)\mu,$$
provided $c_0 \leq (c_5(\zeta_0) \ln 2) / C_2(\zeta_0)$. Hence, either $\mu_{i}^{-}$ or $\mu_i^+$ satisfies the hypothesis of Claim \ref{claim:bootstrapping-1} (with $\zeta=\zeta_0$). Suppose that $\mu_{i}^{-} \leq c_5(\zeta_0) \mu$ (the other case is very similar). Then, by Claim \ref{claim:bootstrapping-1}, we have
\begin{equation}\label{eq:C5bound} \mu_i^- \leq \frac{C_{5}(\zeta_0) (\epsilon-\epsilon_{i}^{+})\mu \ln(1/p)}{\ln\left(\frac{1}{(\epsilon-\epsilon_{i}^{+})\ln(1/p)}\right)},\end{equation}
and so in particular, $\epsilon_i^+ \leq \epsilon$. By applying the induction
hypothesis to $f_{i\to1}$, we obtain
\[
\mu_{p}\left(f_{i\to1}\Delta1_{S'}\right)\leq\frac{C_{1}\epsilon_{i}^{+}\ln(1/p)\mu_{i}^{+}}{\ln\left(\frac{1}{\epsilon_{i}^{+}\ln(1/p)}\right)}
\]
for some subcube $S'=\{x\in\{0,1\}^{[n]\setminus\{i\}}:\ x_{j}=a_j\ \forall j\in T\}$, where $T \subset [n]$ and $a_j \in \{0,1\}$ for each $j \in T$.
Therefore, writing
\[
S:=\{x\in\{0,1\}^{n}:\ x_{j}=a_j\ \forall j\in T,\ x_i = 1\},
\]
 we have
 \begin{align*}
\mu_{p}\left(f\Delta1_{S}\right) & \leq\left(1-p\right)\mu_{i}^{-}+p\mu_{p}\left(f_{i\to1}\Delta1_{S'}\right)\\
 & \leq\left(1-p\right)\left(\frac{C_{5}(\zeta_0)(\epsilon-\epsilon_{i}^{+})\mu \ln(1/p)}{\ln\left(\frac{1}{(\epsilon-\epsilon_{i}^{+})\ln(1/p)}\right)}\right)+\frac{C_{1}\epsilon_{i}^{+}\ln(1/p)p\mu_{i}^{+}}{\ln\left(\frac{1}{\epsilon_{i}^{+}\ln(1/p)}\right)}\\
 & \leq\frac{\left(C_{5}(\zeta_0)\left(\epsilon-\epsilon_{i}^{+}\right)+C_{1}\epsilon_{i}^{+}\right)\mu\ln(1/p)}{\ln\left(\frac{1}{\epsilon\ln(1/p)}\right)}\leq\frac{C_{1}\epsilon\mu\ln(1/p)}{\ln\left(\frac{1}{\epsilon\ln(1/p)}\right)},
\end{align*}
 provided $C_{1}\ge C_{5}(\zeta_0)$, using (\ref{eq:C5bound}). This completes the inductive step, proving the theorem in the case $\zeta_0 \leq p \leq 1/2$, and completing the proof of Theorem \ref{thm:skewed-iso-stability}.

 The inductive proof of Theorem \ref{thm:mon-iso-stability} is very similar indeed, except that the constants are allowed to depend upon $\eta$ (where $\eta$ is as in the statement of Theorem \ref{thm:mon-iso-stability}); we omit the details.
\end{proof}

\section{Sharpness of Theorems \ref{thm:skewed-iso-stability} and \ref{thm:mon-iso-stability}}
\label{sec:examples}

Theorem \ref{thm:skewed-iso-stability} is best possible up to the
values of the absolute constants $c_{0}$ and $C_{1}$. This can be
seen by taking $f=1_{A}$, where
\begin{align*}
A= & \{x\in\{0,1\}^{n}:\ x_{i}=1\ \forall i\in[t]\}\\
 & \cup\{x\in\{0,1\}^{n}:\ x_{i}=1\ \forall i\in[t+s]\setminus\{t\},\ x_{t}=0\}\\
 & \setminus\{x\in\{0,1\}^{n}:\ x_{i}=1\ \forall i\in[t+s]\setminus\{t+1\},\ x_{t+1}=0\},
\end{align*}
 for $s,t\in\mathbb{N}$ with $s\geq3$. Let $0 <p \leq 1/2$. We have $\mu_{p}(A)=p^{t}$,
and
\[
I_{i}[A]=\begin{cases}
p^{t-1} & \textrm{ if }1\leq i\leq t-1;\\
(1-p^{s-1})p^{t-1} & \textrm{ if }i=t;\\
p^{t+s-2} & \textrm{ if }i=t+1;\\
2(1-p)p^{t+s-2} & \textrm{ if }t+2\leq i\leq t+s;\\
0 & \textrm{ if }i>t+s.
\end{cases}
\]
 Hence,
\begin{equation}
\label{eq:influence-equality}
I^{p}[A]=p^{t-1}\left(t+2(s-1)(1-p)p^{s-1}\right).
\end{equation}
 On the other hand, it is easy to see that
\begin{equation}
\label{eq:subcube-bound}
\frac{\mu_{p}(A\Delta S)}{\mu_{p}(A)}=\frac{\mu_{p}(A\Delta S)}{p^{t}}\geq 2(1-p)p^{s-1}
\end{equation}
 for all subcubes $S$, with equality if and only if $S = \{x \in \{0,1\}^n:\ x_i=1\ \forall i \in [t]\}:=C$. Indeed, note that $\mu_p(C \setminus A) = \mu_p(A \setminus C) = (1-p)p^{t+s-1}$. Suppose that $S = \{x \in \{0,1\}^n:\ x_i = a_i\ \forall i \in F\}$, where $F \subset [n]$ and $a_i \in \{0,1\}$ for all $i \in F$. If there exists $i \in F \cap [t]$ such that $a_i=0$, then $S \cap C = \emptyset$ and therefore
 $$\mu_p(A \Delta S) \geq \mu_p(A \setminus S) \geq \mu_p(A \cap C) = p^t - (1-p)p^{t+s-1} > 2(1-p)p^{t+s-1},$$
the last inequality using the fact that $s \geq 3$ and $p(1-p) \leq 1/4$. If $[t] \setminus F \neq \emptyset$, say $j \in [t] \setminus F$, then for any $x \in S \cap C$, we have $x-e_j \in S \setminus C$, and therefore $\mu_p(S \setminus C) \geq \tfrac{1-p}{p}\mu_p(S \cap C) \geq \mu_p(S \cap C)$. Hence,
\begin{equation}\label{eq:sa} \mu_p(S \setminus A) \geq \mu_p(S \setminus C)-\mu_p(A \setminus C) \geq \mu_p(S \cap C)-(1-p)p^{t+s-1}.\end{equation}
On the other hand, we have
\begin{equation} \label{eq:as} \mu_p(A \setminus S) \geq \mu_p(A \cap C) - \mu_p(S \cap C) = p^t - (1-p)p^{t+s-1} - \mu_p(S \cap C).\end{equation}
Summing the inequalities (\ref{eq:sa}) and (\ref{eq:as}), we obtain
$$\mu_p(A \Delta S) \geq p^t -2(1-p)p^{t+s-1} > 2(1-p)p^{t+s-1},$$
the last inequality using the fact that $s \geq 3$ and $p(1-p) \leq 1/4$. Hence, we may assume that $[t] \subset F$ and that $a_i = 1$ for all $i \in [t]$, so in particular $S \subset C$. Suppose that $F \setminus [t] \neq \emptyset$. Then $\mu_p(S) \leq (1-p)\mu_p(C) = (1-p)p^t$, and therefore
 \begin{align*} \mu_p(A \setminus S) & \geq \mu_p((A \cap C) \setminus S) + \mu_p(A \setminus C)\\
 & \geq \mu_p(A \cap C) -  \mu_p(S) + \mu_p(A \setminus C)\\
 & = p^t - (1-p)p^{t+s-1} - \mu_p(S) + (1-p)p^{t+s-1}\\
 & = p^t - \mu_p(S)\\
 & \geq p^{t}-(1-p)p^t\\
 & = p^{t+1}\\
 & > 2(1-p)p^{t+s-1},
 \end{align*}
the last inequality using the fact that $s \geq 3$ and $p(1-p) \leq 1/4$. The only remaining case is $S=C$, where equality holds in (\ref{eq:subcube-bound}).

It follows from (\ref{eq:influence-equality}) and (\ref{eq:subcube-bound}) that if $\epsilon:= 2(s-1)(1-p)p^{s-1}$,
then
\[
pI^{p}[A]=\mu_{p}(A)(\log_{p}(\mu_{p}(A))+\epsilon),
\]
 but
\[
\frac{\mu_{p}(A\Delta S)}{\mu_{p}(A)}\geq\frac{\epsilon}{s-1}:=\delta,
\]
 for all subcubes $S$. We have $(s-1)p^{s-1}\geq\tfrac{\epsilon}{2}$,
so writing $s-1=x/\ln(1/p)$, we get
\[
xe^{-x}\geq\tfrac{1}{2}\epsilon\ln(1/p),
\]
 which implies
\[
x\leq2\ln\left(\frac{1}{\tfrac{1}{2}\epsilon\ln(1/p)}\right),
\]
 or equivalently,
\[
s-1\leq\frac{2}{\ln(1/p)}\ln\left(\frac{1}{\tfrac{1}{2}\epsilon\ln(1/p)}\right).
\]
 Hence,
\[
\delta\geq\frac{\epsilon\ln(1/p)}{2\ln\left(\frac{1}{\frac{1}{2}\epsilon\ln(1/p)}\right)},
\]
 showing that Theorem \ref{thm:skewed-iso-stability} is best possible
up to the value of $C_{1}$. Moreover, we clearly require $\epsilon\ln(1/p)<1$
for the right-hand side of (\ref{eq:conc}) to
be non-negative, so in the statement of Theorem \ref{thm:skewed-iso-stability}, it is necessary that $c_0 < 1$.

Observe that the above family $A$ is not monotone increasing. To prove sharpness for Theorem \ref{thm:mon-iso-stability}, we may take $f = 1_{B}$, where
$$B=\{x\in\{0,1\}^{n}:\ x_{i}=1\ \forall i\in[t]\} \cup\{x\in\{0,1\}^{n}:\ x_{i}=1\ \forall i\in[t+s]\setminus\{t\},\ x_{t}=0\}.$$
 for $s,t\in\mathbb{N}$ with $s\geq3$. Let $0 < p < 1$. We have $\mu_{p}(B)=p^{t}(1+  (1-p)p^{s-1})$,
and
\[
I^p_{i}[B]=\begin{cases}
p^{t-1} + (1-p) p^{t+s-2}& \textrm{ if }1\leq i\leq t-1;\\
(1-p^{s})p^{t-1} & \textrm{ if }i=t;\\
(1-p)p^{t+s-2} & \textrm{ if }t+1\leq i\leq t+s;\\
0 & \textrm{ if }i>t+s.
\end{cases}
\]
Hence,
$$I^p[B] = p^{t-1}(t+((t+s)(1-p)-1)p^{s-1}),$$
and we have
$$\frac{p I^p[B] - \mu_p(B) \log_p(\mu_p(B))}{\mu_p(B)} \leq (s-1)(1-p)p^{s-1} = :\epsilon.$$
On the other hand, we have
\[
\frac{\mu_{p}(B\Delta S)}{\mu_{p}(B)}=\frac{\mu_{p}(B\Delta S)}{p^{t}(1+  (1-p)p^{s-1})}\geq \frac{(1-p)p^{t+s-1}}{p^{t}(1+  (1-p)p^{s-1})} \geq \tfrac{1}{2}(1-p)p^{s-1}: = \delta
\]
 for all subcubes $S$, with equality if and only if $S = \{x\in \{0,1\}^n:\ x_i=1\ \forall i \in [t]\}$, by a very similar argument to that above (for $A$). Similarly to before, we obtain
\[
\delta\geq\frac{\epsilon\ln(1/p)}{4\ln\left(\frac{1-p}{\epsilon\ln(1/p)}\right)}.
\]
Provided $1/e < p < 1$, choosing $s = \lceil 1/\ln(1/p) \rceil +1$ yields 
$$\delta\geq\frac{\epsilon\ln(1/p)}{4\ln\left(\frac{1-p}{\epsilon\ln(1/p)}\right)} = \Omega(\ln(1/(1-p))) \frac{\epsilon\ln(1/p)}{\ln\left(\frac{1}{\epsilon\ln(1/p)}\right)};$$
in this case, writing $p = 1-\eta$, we have $\epsilon = \Theta(1-p) = O(\eta) = O(\eta^2) /\ln(1/p)$. This shows that Theorem \ref{thm:mon-iso-stability} is best possible up to a constant factor depending on $\eta$, and that the statement of Theorem \ref{thm:mon-iso-stability} holds only if $c_0(\eta) = O(\eta^2)$ or $C_1(\eta) = \Omega(\ln(1/\eta))$, so the dependence on $\eta$ cannot be removed.

We note that $B$ also demonstrates the sharpness of Theorem \ref{thm:skewed-iso-stability}, but does not have the nice property of $\log_p(\mu_p(B)) \in \mathbb{N}$, so we think it worthwhile to include both examples.

\section{Isoperimetry via Kruskal-Katona -- Proof of Theorem~\ref{thm:Monotone}, and a new proof of the `full' edge isoperimetric inequality}
\label{sec:KK}

In this section, we use the Kruskal-Katona theorem, the Margulis-Russo lemma and some analytic and combinatorial arguments to prove Theorem~\ref{thm:Monotone}, our biased version of the `full' edge isoperimetric inequality, for monotone increasing sets. We then give the (very short) deduction of Theorem \ref{thm:edge-iso} (the `full' edge isoperimetric inequality) from the $p=1/2$ case of Theorem \ref{thm:Monotone}, hence providing a new proof of the former --- one that relies upon the Kruskal-Katona theorem.

\subsection*{The Margulios-Russo Lemma}
We first recall the useful lemma of Margulis \cite{Margulis} and Russo \cite{Russo}.
\begin{lem}[Margulis, Russo]
\label{lem:MR}Let $\f\subset\pn$ be a monotone increasing family and let $0 < p_0 < 1$. Then
$$\left.\frac{d}{dp} \mu_p(\f)\right|_{p=p_0} = I^{p_0}[\f].$$.
\end{lem}

\subsection*{Lexicographic families in the Cantor space $\p(\mathbb{N})$}

We now give a formal definition of the lexicographic families $\l_{\lambda}$ (described less formally in the Introduction), and analyse some of their properties.

 We define $\l_{0} = \emptyset$ and $\l_{1} = \p(\mathbb{N})$. For any $\lambda\in\left(0,1\right)$, let the binary expansion of $\lambda$ be
\begin{equation}\label{eq:bin-exp-inf} \sum_{j=1}^{\infty}2^{-i_j} = \lambda \end{equation}
where $1 \leq i_1 < i_2 < \ldots$ (if the binary expansion is infinite),
or
\begin{equation} \label{eq:bin-exp-fin} \sum_{j=1}^{N}2^{-i_j} = \lambda \end{equation}
where $1 \leq i_1 < i_2 < \ldots < i_{N}$ (if the binary expansion is finite), and define
$$\l_{\lambda} = \bigcup_{j} \{S \subset \mathbb{N}:\ S \cap [i_j] = [i_j] \setminus \{i_k:\ k < j\}\} \subset \mathcal{P}(\mathbb{N}).$$
Equivalently, let $T = \{i_1,i_2,\ldots\}$ be the set whose characteristic vector corresponds to the binary expansion of $\lambda$, and let $\l_{\lambda} = \{S \subset [n]:\ S \geq \mathbb{N} \setminus T\}$ be the initial segment of the lexicographic ordering on $\p(\mathbb{N})$ ending at $\mathbb{N} \setminus T$.

Note that if the binary expansion of $\lambda$ is finite, i.e.\ $2^n \lambda \in \mathbb{N} \cup \{0\}$ for some $n \in \mathbb{N}$, then $\l_{\lambda} = \l \times \mathcal{P}(\mathbb{N} \setminus [n])$, where $\l \subset \p([n])$ is the lexicographic family of size $2^n \lambda$.

We identify $\mathcal{P}(\mathbb{N})$ with the Cantor space $\{0,1\}^{\mathbb{N}}$, in the natural way. We let $\Sigma$ be the $\sigma$-algebra on $\p(\mathbb{N})$ generated by $\cup_{n \in \mathbb{N}} \pn$. By the countable unions property of $\sigma$-algebras, it is clear that $\l_\lambda \in \Sigma$ for any $\lambda \in [0,1]$.

By the Kolmogorov Extension theorem (see \cite{kol}, or e.g.\ \cite{tao} for a more modern exposition), there exists a unique probability measure $\mu_p^{(\mathbb{N})}$ on $(\{0,1\}^{\mathbb{N}},\Sigma)$ such that
$$\mu_p^{(\mathbb{N})}(A_1 \times A_2 \times \ldots \times A_n \times \{0,1\} \times \{0,1\} \times \ldots) = \mu_p^{(n)}(A_1 \times A_2 \times \ldots \times A_n)$$
for all $n \in \mathbb{N}$ and all $A_1,\ldots,A_n \subset \{0,1\}$. We may call this measure the {\em $p$-biased product measure} on $\{0,1\}^{\mathbb{N}}$.

Abusing notation slightly, we write $\mu_p = \mu_p^{(\mathbb{N})}$ when the underlying space $\{0,1\}^{\mathbb{N}}$ is understood.

If $f: \{0,1\}^{\mathbb{N}} \to \{0,1\}$ is $\Sigma$-measurable, we define influence of the $i$th coordinate on $f$ by
\[I_i^p[f] := \Pr_{x \sim \mu_p}[f(x) \neq f(x \oplus e_i)]\]
and we define the total influence of $f$ by
$$I^p[f] := \sum_{i=1}^{\infty} I_i^p[f].$$

We remark that there exist $\Sigma$-measurable functions $f: \{0,1\}^{\mathbb{N}} \to \{0,1\}$ such that $I^p[f] = \infty$.
However, the families $\l_{\lambda}$ are better behaved, as we will shortly see.

Clearly, by the countable additivity of $\mu_p$, we have
\begin{equation}\label{eq:countable-add} \mu_p(\l_{\lambda}) = \sum_{j}p^{i_j-j+1}(1-p)^{j-1},\end{equation}
where the $(i_j)$ define the binary expansion of $\lambda$, as in (\ref{eq:bin-exp-inf}) or (\ref{eq:bin-exp-fin}).

 It is helpful to analyse the families $\l_{\lambda}$ using the families $(\l_{\lfloor \lambda 2^n \rfloor /2^n})_{n \in \mathbb{N}}$, which depend upon only finitely many coordinates. To this end, for each $\lambda \in [0,1]$ and each $n \in \mathbb{N}$, we define $\l_{\lambda}(n): = \l_{\lfloor \lambda 2^n \rfloor /2^n}$. For brevity, if $p \in (0,1)$ is fixed, we write $r = r(p): = \max\{p,1-p\}$, and if $\lambda \in [0,1]$ is fixed, we write $\l := \l_{\lambda}$ and $\l(n) := \l_{\lambda}(n) = \l_{\lfloor \lambda 2^n \rfloor 2^n}$ for each $n \in \mathbb{N}$. Observe that for any $\lambda \in [0,1]$, we have $\l(n) \subset \l(n+1) \subset \l$ for all $n \in \mathbb{N}$.

 \begin{claim}
\label{claim:measures-converge}
Let $0 < p < 1$ and let $0 \leq \lambda \leq 1$. Then
$$\mu_p(\l \setminus \l(n)) \leq \frac{r^{n+1}}{1-r}.$$
\end{claim}
\begin{proof}
We may assume that $0 < \lambda < 1$. Let the binary expansion of $\lambda$ be
$$\lambda = \sum_{j} 2^{-i_j},$$
where $1 \leq i_1 < i_2 < \ldots$, so that by definition,
$$\l = \l_{\lambda} = \bigcup_{j} \{S \subset \mathbb{N}:\ S \cap [i_j] = [i_j] \setminus \{i_k:\ k < j\}\} \subset \mathcal{P}(\mathbb{N}).$$
Observe that for each $n \in \mathbb{N}$, we have
$$\l(n) = \bigcup_{j:\ i_j \leq n} \{S \subset \mathbb{N}:\ S \cap [i_j] = [i_j] \setminus \{i_k:\ k < j\}\}.$$
For brevity, write $C_j : = \{S \subset \mathbb{N}:\ S \cap [i_j] = [i_j] \setminus \{i_k:\ k < j\}\}$ for each $j$; then $C_j$ is a subcube whose set of fixed coordinates is $[i_j]$, for each $j$, and we have
$$\l = \bigcup_j C_j,\quad \l(n) = \bigcup_{j:\ i_j \leq n}C_j.$$
Hence,
$$\mu_p(\l \setminus \l(n)) = \sum_{j:\ i_j > n} \mu_p(C_j) \leq r^{n+1}+r^{n+2} +\ldots \leq \frac{r^{n+1}}{1-r},$$
since the subcube $C_j$ has $i_j$ fixed coordinates, for all $j$.
 \end{proof}

 It follows from Claim \ref{claim:measures-converge} that
 \begin{equation} \label{eq:limit-measures-2} \mu_{p}\left(\l_{\lambda}\right) = \lim_{n \to \infty} \mu_{p}\left(\l_{\lfloor \lambda 2^n \rfloor/2^n}\right),
 \end{equation}
where we can regard $\l_{\lfloor \lambda 2^n \rfloor/2^n}$ either as a subset of $\p(\mathbb{N})$ (with $\mu_p = \mu_p^{(\mathbb{N})}$) or as a subset of $\p([n])$ (with $\mu_p = \mu_p^{(n)}$, the $p$-biased measure on $\pn$); the two measures coincide on families depending only upon the first $n$ coordinates. (Alternatively, it is easy to deduce (\ref{eq:limit-measures-2}) from (\ref{eq:countable-add}).)

In order to analyse $I^p[\l_{\lambda}]$, we need some further observations. If $\a \subset \p(\mathbb{N})$, we write $\a_i^+ = \{S \setminus \{i\}:\ i \in S,\ S \in \a\} \subset \p(\mathbb{N} \setminus \{i\})$, and we write $\a_i^- = \{S \in \a:\ i \notin S\} \subset \p(\mathbb{N} \setminus \{i\})$. If $i \in \mathbb{N}$, we define the `projected' $\sigma$-algebra
$$\Sigma_i := \{\{S \setminus \{i\}:\ S \in \f\}:\ \f \in \Sigma\} \subset \p(\mathbb{N} \setminus \{i\}),$$
and we equip $(\p(\mathbb{N} \setminus \{i\}),\Sigma_i)$ with the natural product measure $\mu_p^{(\mathbb{N}\setminus \{i\})}$ induced by $\mu_p^{(\mathbb{N})}$, i.e.\ for all $\g \in \Sigma_i$,
$$\mu_p^{(\mathbb{N} \setminus \{i\})}(\g) := \mu_p^{(\mathbb{N})}(\{F \subset \mathbb{N}:\ F \setminus \{i\} \in \g\}).$$
It is easily checked that if $\a\in \Sigma$, then $\a_i^+,\a_i^- \in \Sigma_i$, and if moreover $\a$ is monotone increasing, then
$$I_i^p[\a] = \mu_p^{(\mathbb{N} \setminus \{i\})}(\a_i^+ \setminus \a_i^-) = \mu_p^{(\mathbb{N} \setminus \{i\})}(\a_i^+) - \mu_p^{(\mathbb{N} \setminus \{i\})}(\a_i^-).$$
For brevity, we will write $\mu_p = \mu_p^{(\mathbb{N} \setminus \{i\})}$ when the underlying space $\{0,1\}^{\mathbb{N} \setminus \{i\}}$ is clear from the context.

We can now prove the following.

\begin{claim}
\label{claim:high-coords}
Let $0 < p < 1$, let $0 \leq \lambda \leq 1$ and let $i \in \mathbb{N}$. Then $I_i^{p}[\l] \leq r^{i}/(1-r)^2$.
\end{claim}
\begin{proof}
Since $\l=\l_{\lambda}$ is monotone increasing, we have
$$I_i^p[\l] = \mu_p(\l_i^+ \setminus \l_i^-).$$
If $S \in \l_i^+\setminus \l_i^-$, then $S \cup \{i\} \in \l \setminus \l(i-1)$, since $\l(i-1)$ depends only upon the first $i-1$ coordinates. Since $\mu^{(\mathbb{N})}_p(\{S \cup \{i\}\}) = p \mu_p^{(\mathbb{N} \setminus \{i\})} (\{S\})$ for each such $S$, we have
$$p \mu_p^{(\mathbb{N} \setminus \{i\})}(\l_i^+\setminus \l_i^-) \leq \mu^{(\mathbb{N})}_p(\l \setminus \l(i-1)).$$
By Claim \ref{claim:measures-converge}, we have
$\mu_p(\l \setminus \l(i-1)) \leq r^{i}/(1-r)$, and therefore
$$I_i^{p}[\l] = \mu_p^{(\mathbb{N} \setminus \{i\})}(\l_i^+ \setminus \l_i^-) \leq \frac{\mu_p^{(\mathbb{N})}(\l \setminus \l(i-1))}{p} \leq \frac{r^{i}}{p(1-r)} \leq  \frac{r^{i}}{(1-r)^2},$$
as required.
\end{proof}

It follows from Claim \ref{claim:high-coords} that $I^p[\l_{\lambda}] \leq \sum_{i=1}^{\infty} r^{i}/(1-r)^2 = r/(1-r)^3 < \infty$, for any $p \in (0,1)$ and any $\lambda \in [0,1]$.

\begin{claim}
\label{claim:low-coords}
Let $0 < p < 1$ and let $0 \leq \lambda \leq 1$. Then for each $i \in \mathbb{N}$, we have
$$\left|I_i^{p}[\l] - I_i^p[\l(n)]\right| \leq \frac{r^{n}}{(1-r)^2}.$$
\end{claim}
\begin{proof}
Observe that for any monotone increasing $\a,\b \in \Sigma$ with $\b \subset \a$, and any $i \in \mathbb{N}$, we have
\begin{align*} \left|I_i^p[\a] - I_i^p[\b]\right| & = \left|(\mu_p(\a_i^+) - \mu_p(\a_i^-)) - (\mu_p(\b_i^+)-\mu_p(\b_i^-))\right|\\
& = |(\mu_p(\a_i^+) - \mu_p(\b_i^+)) - (\mu_p(\a_i^-)-\mu_p(\b_i^-))|\\
& \leq \max\{\mu_p(\a_i^+) - \mu_p(\b_i^+), \mu_p(\a_i^-)-\mu_p(\b_i^-)\}\\
& = \max\{\mu_p(\a_i^+ \setminus \b_i^+), \mu_p(\a_i^-\setminus \b_i^-)\}\\
& \leq \frac{\mu_p(\a \setminus \b)}{\min\{p,1-p\}}\\
&= \frac{\mu_p(\a \setminus \b)}{1-r}.
\end{align*}
Applying this with $\a = \l$ and $\b = \l(n)$, and using Claim \ref{claim:measures-converge}, yields
$$\left|I_i^{p}[\l] - I_i^p[\l(n)]\right| \leq \frac{r^{n+1}}{(1-r)^2}\quad \forall i \in \mathbb{N},$$
as required.
\end{proof}

The two claims above yield the following.
\begin{lem}
\label{lem:unif-limit-infs}
$$|I^{p}[\l] - I^p[\l(n)]| \leq \frac{n r^n}{(1-r)^3}.$$
\end{lem}

\begin{proof}
Since $\l(n)$ depends only upon the first $n$ coordinates, we have $I^p_i[\l(n)] = 0$ for all $i > n$. Hence,
\begin{align*}
|I^{p}[\l] - I^p[\l(n)]| & \leq \left|\sum_{i=1}^{n} (I_i^{p}[\l] - I^p_i[\l(n)]) \right| +  \left|\sum_{i=n+1}^{\infty} (I_i^{p}[\l] - I_i^p[\l(n)]) \right|\\
& = \left|\sum_{i=1}^{n} (I_i^{p}[\l] - I_i^p[\l(n)]) \right| +  \sum_{i=n+1}^{\infty} I_i^{p}[\l]\\
& \leq \sum_{i=1}^{n} |I_i^{p}[\l] - I^p_i[\l(n)]| + \sum_{n+1}^{\infty} I_i^{p}[\l]
\leq n \frac{r^{n}}{(1-r)^2} + \sum_{i=n+1}^{\infty} \frac{r^{i}}{(1-r)^2}\\
&= \frac{((1-r)n+r)r^n}{(1-r)^3} \leq \frac{n r^n}{(1-r)^3},
\end{align*}
where the third inequality uses Claim \ref{claim:low-coords} to bound the first sum and Claim \ref{claim:high-coords} to bound the second.
\end{proof}

Lemma \ref{lem:unif-limit-infs} implies that
\begin{equation} \label{eq:limit-influences-2} I^{p}\left[\l_{\lambda}\right]=\lim_{n \to \infty} I^{p}\left[\l_{\lfloor \lambda 2^n \rfloor/2^{n}}\right],\end{equation}
where we can regard $\l_{\lfloor \lambda 2^n \rfloor/2^n}$ either as a subset of $\p(\mathbb{N})$ or as a subset of $\p([n])$; the two relevant notions of influence coincide on families depending only upon the first $n$ coordinates.

Lemma \ref{lem:unif-limit-infs} also implies that the statement of the Margulis-Russo lemma holds for $\l_{\lambda}$:
\begin{lem}
\label{lem:MR-inf}
If $0 < p_0 < 1$ and $0 \leq \lambda \leq 1$, then the function $p \mapsto \mu_p(\l_{\lambda})$ is differentiable at $p_0$, with
$$\left.\frac{d}{dp} \mu_p(\l_{\lambda})\right|_{p=p_0} = I^{p_0}[\l_{\lambda}].$$
\end{lem}
\begin{proof} We may assume that $0 < \lambda < 1$. Fix such a $\lambda$. Define the function $g:(0,1) \to [0,1];\ g(p) = \mu_p(\l)$, and for each $n \in \mathbb{N}$, define a function $g_n:(0,1) \to [0,1];\ g_n(p) = \mu_p(\l(n))$. By (\ref{eq:limit-measures-2}), $g_n(p) \to g(p)$ as $n \to \infty$, for any $p \in (0,1)$. By the Margulis-Russo lemma, $g_n'(p) = I^p[\l(n)]$ for each $n \in \mathbb{N}$, since for each $n \in \mathbb{N}$, the family $\l(n) \subset \p(\mathbb{N})$ can be viewed as a subset of $\pn$, with the respective definitions of total influence coinciding. Moreover, by Lemma \ref{lem:unif-limit-infs}, provided $\eta \leq p \leq 1-\eta$ where $\eta >0$, we have
\begin{equation}\label{eq:lim-inf} |I^p[\l] - g_n'(p)|= |I^p[\l] - I^p[\l(n)]| \leq \frac{n(1-\eta)^n}{\eta^3} \to 0\quad \text{as } n \to \infty,\end{equation}
so $g_n'$ converges uniformly to the function $p \mapsto I^p[\l]$ on the interval $[\eta,1-\eta]$, for any $\eta >0$. It follows from the Differentiable Limit theorem that $g$ is differentiable, and that for any $p_0 \in (0,1)$ we have
$$\left.\frac{d}{dp} \mu_p(\l)\right|_{p=p_0} = g'(p_0) = \lim_{n \to \infty} g_n'(p_0) = \lim_{n \to \infty} I^{p_0}[\l(n)] = I^{p_0}[\l],$$
using (\ref{eq:lim-inf}) again for the last equality. This proves the lemma.
\end{proof}

We also need the following claims.

\begin{claim}
\label{claim:power}
Let $0 < p < 1$ and let $\f \in \Sigma$. Then
$$\mu_p(\f) \leq (\mu_{1/2}(\f))^{\log_{1/2}(r)}.$$
\end{claim}
\begin{proof}
Let $0 <p < 1$. Since the algebra of sets
$$\{\f \times \p(\mathbb{N} \setminus [n]):\ n \in \mathbb{N},\ \f \subset \p([n])\}$$
is dense in the probability space $(\p(\mathbb{N}),\Sigma,\mu_p)$ and in the probability space $(\p(\mathbb{N}),\Sigma,\mu_{1/2})$, it suffices to prove the claim when $\f \subset \p([n])$ for some $n \in \mathbb{N}$.

Let $S \subset [n]$. Then
\begin{align*} \mu_p(\{S\}) = p^{|S|}(1-p)^{n-|S|}
 \leq r^n
 = (2^{-n})^{\log_{1/2}(r)}
 = (\mu_{1/2}(\{S\}))^{\log_{1/2}(r)}.
\end{align*}
Hence, for any $\f \subset \p([n])$, we have
\begin{align*} 
\mu_p(\f) & = \sum_{S \in \f} \mu_p(\{S\})
 \leq \sum_{S \in \f} (\mu_{1/2}(\{S\}))^{\log_{1/2}(r)} \\
&\leq \left( \sum_{S \in \f} \mu_{1/2}(\{S\})\right)^{\log_{1/2}(r)}
 = (\mu_{1/2}(\f))^{\log_{1/2}(r)},
\end{align*}
the last inequality using the fact that $\log_{1/2}(r) \geq 1$.
\end{proof}

\begin{claim}
\label{claim:continuous}
Let $0 < p < 1$. The function $f_p: [0,1] \to [0,1];\ \lambda \mapsto \mu_p(\l_{\lambda})$ is continuous.
\end{claim}
\begin{proof}
Let $0 < p < 1$. Observe that $\mu_{1/2}(\l_{\lambda}) = \lambda$ for all $\lambda \in [0,1]$, and that since the families $\l_{\lambda}$ are nested, $f_p$ is monotone increasing. Let $0 \leq \lambda < \lambda' \leq 1$. The family $\l_{\lambda'}\setminus \l_{\lambda}$ is clearly $\Sigma$-measurable, and we have
\begin{align*} f_p(\lambda')-f_p(\lambda) & = \mu_p(\l_{\lambda'}) - \mu_{p}(\l_{\lambda})
= \mu_p(\l_{\lambda'} \setminus \l_{\lambda})\\
& \leq (\mu_{1/2}(\l_{\lambda'} \setminus \l_{\lambda}))^{\log_{1/2}(r)}
= (\lambda'-\lambda)^{\log_{1/2}(r)}\\
& \to 0 \quad \text{as }\lambda' - \lambda \to 0,
\end{align*}
using Claim \ref{claim:power} for the last inequality. It follows that $f_p$ is continuous, as required.
\end{proof}

We now know that for each $p \in (0,1)$, the function $f_p:\ \lambda \mapsto \mu^{(\mathbb{N})}_p(\l_{\lambda})$ is continuous and monotone increasing, with $f_p(0)=0$ and $f_p(1)=1$. Hence, by the intermediate value theorem, for any $p \in (0,1)$ and any $x \in [0,1]$, there exists $\lambda \in [0,1]$ such that $\mu^{(\mathbb{N})}_p(\l_{\lambda})=x$. In particular, for each $n \in \mathbb{N}$ and each $\f \subset \pn$, there exists $\lambda \in [0,1]$ such that $\mu^{(\mathbb{N})}_p(\l_{\lambda})=\mu^{(n)}_p(\f)$, where $\mu_p^{(n)}$ denotes the $p$-biased measure on $\pn$, i.e.\ there always exists a $\lambda \in [0,1]$ as in the hypothesis of Theorem \ref{thm:Monotone}.

\subsection*{The Kruskal-Katona theorem, and some applications}
In our proof of Theorem \ref{thm:Monotone}, we will also use the well-known Kruskal-Katona theorem \cite{Katona66,Kruskal63}. To state it, we need some more notation. For $k,n \in \mathbb{N} \cup \{0\}$ with $0 \leq k \leq n$, we write $[n]^{(k)} := \{S \subset [n]:\ |S|=k\}$. For a family $\f \subset \p([n])$ and $0 \leq k \leq n$, we write $\f^{(k)}: = \mathcal{F} \cap [n]^{(k)}$. If $k < n$ and $\mathcal{A} \subset [n]^{(k)}$, we write $\partial^{+}(\mathcal{A}) := \{B \in [n]^{(k+1)}:\ A \subset B\ \text{for some }A \in \mathcal{A}\}$ for the {\em upper shadow} of $\mathcal{A}$, and if $1 \leq i \leq n-k$, we write $\partial^{+(i)}(\mathcal{A}) := \{B \in [n]^{(k+i)}:\ A \subset B \text{ for some }A \in \mathcal{A}\}$ for its $i$th iterate. We define the {\em lexicographic ordering on $[n]^{(k)}$} to be the restriction to $[n]^{(k)}$ of the lexicographic ordering on $\pn$, i.e.\ if $S,T \in [n]^{(k)}$, then $S > T$ iff $\min(S \Delta T) \in S$. If $0 \leq m \leq \binom{n}{k}$, we define $\l_{(n,k,m)}$ to be the size-$m$ initial segment of the lexicographic ordering on $[n]^{(k)}$, i.e.\ the $m$ largest elements of $[n]^{(k)}$ with respect to the lexicographic ordering. Clearly, for any $0 \leq m \leq \binom{n}{k}$, we have $\l_{(n,k,m)} = \l \cap [n]^{(k)}$ for some initial segment $\l$ of the lexicographic ordering on $\pn$.

We can now state the Kruskal-Katona theorem.
\begin{thm}[Kruskal-Katona theorem]
Let $1 \leq k < n$, and let $\f \subset [n]^{(k)}$. Then $|\partial^{+}(\f)| \geq |\partial^+(\l_{(n,k,|\f|)}|$.
\end{thm}
\noindent We need the following straightforward corollary.
\begin{cor}
\label{cor:KK}Let $n>k_{0}>k\geq j\geq 1$ with $n-k_0 \geq j$, suppose that $\l \subset \pn$ is a lexicographically ordered family depending only upon the coordinates in $[j]$, and let $\f\subset\pn$ be a monotone increasing family with $|\f^{\left(k_{0}\right)}| \leq |\l^{\left(k_{0}\right)}|$. Then $|\f^{\left(k\right)}|\le|\l^{\left(k\right)}|$.
\end{cor}
\begin{proof}
Suppose that $|\f^{\left(k_{0}\right)}| \leq |\l^{\left(k_{0}\right)}|$, and assume for a contradiction that $|\f^{\left(k\right)}|\geq |\l^{\left(k\right)}|+1$. Let $\tilde{\l} \subset \pn$ be the minimal lexicographically ordered family such that $|\f^{\left(k\right)}|= |\tilde{\l}^{\left(k\right)}|$; then $\tilde{\l}^{(k)} \setminus \l^{(k)} \neq \emptyset$. Choose $S \in \tilde{\l}^{(k)} \setminus \l^{(k)}$. Since $k_0 \leq n-j$, there exists $S' \supset S$ such that $|S'| = k_0$ and $(S' \setminus S) \cap [j] = \emptyset$, and therefore $S' \in \partial^{+(k_0-k)}(\tilde{\l}^{(k)}) \setminus \l$. Since $j \leq k$ and $\l$ depends only upon the coordinates in $[j]$, we have $\l^{(k_0)} = \partial^{+(k_0-k)}(\l^{(k)}) \subset \partial^{+(k_0-k)}(\tilde{\l}^{(k)})$. It follows that $|\partial^{+(k_0-k)}(\tilde{\l}^{(k)})| >  |\l^{\left(k_0\right)}|$. By repeated application of the Kruskal-Katona theorem, since $|\f^{\left(k\right)}|= |\tilde{\l}^{\left(k\right)}|$ and $\f$ is monotone increasing, we have
$$|\f^{(k_0)}|\geq|\partial^{+(k_0-k)}(\f^{(k)})| \geq |\partial^{+(k_0-k)}(\tilde{\l}^{(k)})| >  |\l^{\left(k_0\right)}|,$$
a contradiction.
\end{proof}

\noindent This implies the following, by a standard application of the method of Dinur-Safra \cite{Dinur-Safra} / Frankl-Tokushige \cite{FT03}, known as `going to infinity and back'. (We present the proof, for completeness.)
\begin{cor}
\label{thm:KK cor}Let $0 < q < p < 1$, let $0 < \lambda < 1$, and let
$\f\subset\pn$ be a monotone increasing family with $\mu_p\left(\f\right)\leq\mu_{p}\left(\l_{\lambda}\right)$.
Then $\mu_{q}\left(\f\right)\le\mu_{q}\left(\l_{\lambda}\right)$.
\end{cor}

\begin{proof}
Let $\f\subset\pn$ be a monotone increasing family with $\mu_p\left(\f\right)\leq\mu_{p}\left(\l_{\lambda}\right)$, and suppose for a contradiction that $\mu_{q}\left(\f\right) > \mu_{q}\left(\l_{\lambda}\right)$. By Claim \ref{claim:continuous}, there exists $\lambda'> \lambda$ such that $\mu_{q}\left(\f\right) >\mu_{q}\left(\l_{\lambda'}\right)$. By (\ref{eq:limit-measures-2}), there exists $m\geq n$ such that
$$\mu_p^{(m)}(\l_{\lambda'}\cap \mathcal{P}([m])) > \mu_{p}\left(\l_{\lambda}\right),\quad \mu_q^{(m)}(\l_{\lambda'}\cap \mathcal{P}([m])) > \mu_{q}\left(\l_{\lambda}\right).$$
Define $\l' = \l_{\lambda'}\cap \mathcal{P}([m]) \subset \mathcal{P}([m])$; then
$$\mu_p(\l') > \mu_{p}\left(\l_{\lambda}\right),\quad \mu_q(\l') > \mu_{q}\left(\l_{\lambda}\right).$$
 Now, for any family $\mathcal{G}\subset \mathcal{P}\left(\left[n\right]\right)$
and any $N\in \mathbb{N}$ with $N \geq n$, we define
$$\mathcal{G}_{N} : = \{A \subset [N]:\ A \cap [n] \in \mathcal{G}\}.$$
It is easily checked that for any $\mathcal{G} \subset \mathcal{P}([n])$ and any $p \in (0,1)$, we have
$$\mu_p(\mathcal{G}) = \lim_{N \to \infty} \frac{|(\mathcal{G}_{N})^{(\lfloor pN \rfloor)}|}{{N \choose \lfloor pN \rfloor}}.$$
In particular, we have
$$\mu_q(\mathcal{F}) = \lim_{N \to \infty} \frac{|(\mathcal{F}_{N})^{(\lfloor qN \rfloor)}|}{{N \choose \lfloor qN \rfloor}}$$
and
$$\mu_q(\l') = \lim_{N \to \infty} \frac{|(\mathcal{L}'_{N})^{(\lfloor qN \rfloor)}|}{{N \choose \lfloor qN \rfloor}}.$$
Since $\mu_q(\f) > \mu_q(\l_{\lambda'}) \geq \mu_q(\l')$, for all $N$ sufficiently large (depending on $q$ and $m$), we have
$$|(\mathcal{F}_{N})^{(\lfloor qN \rfloor)}| > |(\mathcal{L}'_{N})^{(\lfloor qN \rfloor)}|.$$
Since $\mathcal{L}'_{N}$ depends only upon the coordinates in $[m]$, and is a lexicographic family, it follows from Corollary \ref{cor:KK} that if $N$ is sufficiently large depending on $p,q$ and $m$, then
$$|(\mathcal{F}_{N})^{(\lfloor pN \rfloor)}| > |(\mathcal{L}'_{N})^{(\lfloor pN \rfloor)}|.$$
Since
$$\mu_p(\mathcal{F}) = \lim_{N \to \infty} \frac{|(\mathcal{F}_{N})^{(\lfloor pN \rfloor)}|}{{N \choose \lfloor pN \rfloor}}$$
and
$$\mu_p(\l') = \lim_{N \to \infty} \frac{|(\mathcal{L}'_{N})^{(\lfloor pN \rfloor)}|}{{N \choose \lfloor pN \rfloor}},$$
it follows that $\mu_p(\mathcal{F}) \geq \mu_p(\mathcal{L}') > \mu_p(\mathcal{L}_{\lambda})$, a contradiction.
\end{proof}

\medskip

\noindent Now we are ready to prove Theorem~\ref{thm:Monotone}.

\begin{proof}[Proof of Theorem~\ref{thm:Monotone}]
Let $\f$ be a family that satisfies the assumptions of the theorem. Note that by Lemmas \ref{lem:MR} and \ref{lem:MR-inf}, for any $p_0 \in (0,1)$, we have $\frac{d}{dp}\mu_{p}\left(\l_{\lambda}\right)|_{p=p_0}=I^{p_0}\left[\l_{\lambda}\right]$ and
$\frac{d}{dp}\mu_{p}(\f)|_{p=p_0}=I^{p_0}[\f]$. By Corollary \ref{thm:KK cor}, $\mu_{q}\left(\f\right)\le\mu_{q}\left(\l_{\lambda}\right)$
for any $q\le p$. Therefore,
\[
I^{p}\left[\f\right]=\lim_{q\to p}\frac{\mu_{p}\left(\f\right)-\mu_{q}\left(\f\right)}{p-q}\ge\lim_{q\to p}\frac{\mu_{p}\left(\l_{\lambda}\right)-\mu_{q}\left(\l_{\lambda}\right)}{p-q}=I^{p}\left[\l_{\lambda}\right],
\]
as desired.
\end{proof}

\subsection*{The deduction of Theorem~\ref{thm:edge-iso} from Theorem \ref{thm:Monotone}}
This is a standard (and short) `monotonization' argument. We include it for completeness.

For $i \in [n]$, the {\em $i$th monotonization operator} $\m_i:\pn \to \pn$ is defined as follows. (See e.g. \cite{KKL}.) If $\f \subset \pn$, then for each $S \in \f$ we define
$$\m_i(S) = \begin{cases} S \cup \{i\} & \text{if } S \in \mathcal{F},\ i \notin S \text{ and } S \cup \{i\} \notin \f,\\
S & \text{otherwise},\end{cases}$$
and we define $\m_i(\f) = \{\m_i(S):\ S \in \f\}$. It is well-known, and easy to check, that for any $\f \subset \pn$, we have $|\m_i(\f)| = |\f|$ and
$$I^{1/2}_{j}\left[\m_{i}\left(\f\right)\right]\le I^{1/2}_{j}\left[\f\right]\quad \forall j \in [n];$$
summing over all $j$ we obtain
$$I^{1/2}\left[\m_{i}\left(\f\right)\right]\le I^{1/2}\left[\f\right].$$
Observe that the $\m_{i}$'s transform a family to a monotone increasing one, in the sense that for any $\f \subset \pn$, the family $\g:=\m_{1}\circ\cdots\circ\m_{n}\left(\f\right)$ is monotone increasing; note also that $|\g| = |\f|$ and $I^{1/2}[\g]\le I^{1/2}[\f]$.

Now let $\f \subset \pn$, and let $\l_{\lambda} \subset \pn$ be a lexicographic family with $|\l_{\lambda}|=|\f|$. Let $\g = \m_{1}\circ\cdots\circ\m_{n}\left(\f\right)$; then $|\g|=|\f|$, $I^{1/2}[\g] \leq I^{1/2}[\f]$, and $\g$ is monotone increasing. By Theorem~\ref{thm:Monotone}, we have $I^{1/2}[\g] \geq I^{1/2}[\l_{\lambda}]$, and therefore $I^{1/2}[\f] \geq I^{1/2}[\g] \geq I^{1/2}[\l_{\lambda}]$, proving Theorem \ref{thm:edge-iso}.

\begin{remark}
\label{remark:anti}
We observe that the statement of Theorem \ref{thm:edge-iso} does not hold for arbitary (i.e., non-monotone) families $\f$, if $p \neq 1/2$. Indeed, let $\f = \{S \subset [n]:\ 1 \notin S\}$, and let $p \in (0,1) \setminus \{1/2\}$; then $\mu_p(\f) = 1-p$ and $I^p[\f] = 1$. Since the function $f_p: \lambda \mapsto \mu_p(\l_{\lambda})$ is continuous (by Claim \ref{claim:continuous}) with $f_p(0)=0$ and $f_p(1)=1$, there exists $\lambda \in (0,1)$ such that $\mu_p(\l_{\lambda}) = 1-p$. Write $\l = \l_{\lambda}$, and as before, for each $n \in \mathbb{N}$, write $\l(n) = \l_{\lfloor \lambda 2^n \rfloor / 2^n}$. Then we may view $\l(n)$ as a subset of $\p([n])$, for each $n \in \mathbb{N}$. We have $\mu_p(\l(n)) \to \mu_p(\l) = 1-p$ as $n \to \infty$, by (\ref{eq:limit-measures-2}).

First suppose that $1/2 < p < 1$. By Theorem \ref{thm:edge-iso-biased}, and since $\l(n)$ is monotone increasing with $\mu_p(\l(n)) \to 1-p$ as $n \to \infty$, we have
$$p I^p[\l(n)] \geq \mu_p(\l(n)) \log_p(\mu_p(\l(n))) \to (1-p) \log_p(1-p)\quad \text{as } n \to \infty.$$
It follows from (\ref{eq:limit-influences-2}) that
$$p I^p[\l] \geq  (1-p) \log_p(1-p) > p,$$
the last inequality using Claim \ref{claim on alpha(p)} and the fact that $p > 1/2$. Hence, $I^p[\l] > 1 = I^p[\f]$.

Now suppose that $0 < p < 1/2$. Note that $\l(n)^* \subset \pn$ is monotone increasing with $\mu_{1-p}(\l(n)^*) = 1-\mu_p(\l(n))$ and $I^{1-p}[\l(n)^*] = I^p[\l(n)]$. By Theorem \ref{thm:edge-iso-biased}, and since $\l(n)^*$ is monotone increasing, we have
\begin{align*} (1-p) I^{p}[\l(n)] = (1-p)I^{1-p}[\l(n)^*] & \geq \mu_{1-p}(\l(n)^*) \log_{1-p}(\mu_{1-p}(\l(n)^*))\\
& \to p \log_{1-p}(p)
\end{align*}
as $n \to \infty$, since $\mu_{1-p}(\l(n)^*) = 1-\mu_p(\l(n)) \to p$ as $n \to \infty$. It follows from (\ref{eq:limit-influences-2}) that
$$(1-p) I^p[\l] \geq  p \log_{1-p}(p) > 1-p,$$
the last inequality using Claim \ref{claim on alpha(p)} and the fact that $p < 1/2$. Hence, $I^p[\l] > 1 = I^p[\f]$.
\end{remark}
\section{Open Problems}
\label{sec:open}

A natural open problem is to obtain a $p$-biased edge-isoperimetric inequality for arbitrary (i.e., not necessarily monotone increasing) families, which is sharp for all values of the $p$-biased measure. This is likely to be difficult, as there is no nested sequence of extremal families. Indeed, it is easily checked that if $p < 1/2$, the unique families $\mathcal{F} \subset \pn$ with $\mu_p(\f) = p$ and minimal $I_p[\f]$ are the dictatorships, whereas the unique families $\g \subset \pn$ with $\mu_p(\g) = 1-p$ and minimal $I_p[\g]$ are the antidictatorships; clearly, none of the former are contained in any of the latter.

Another natural problem is to obtain a sharp stability version of our `full' biased edge isoperimetric inequality for monotone increasing families (i.e.,
Theorem~\ref{thm:Monotone}). This would generalise (the monotone case of) Theorem \ref{thm:full-stability}, our sharp stability version of the `full' edge isoperimetric inequality. It seems likely that the proof in \cite{LOL} can be extended to the biased case using the methods of the current paper, but the resulting proof is expected to be rather long and complex.

Finally, it is highly likely that the values of the absolute constants in Theorem \ref{thm:skewed-iso-stability}, and of the constants depending upon $\eta$ in Theorem \ref{thm:mon-iso-stability}, could be substantially improved. Note for example that Theorem \ref{thm:skewed-iso-stability} applies only to Boolean functions whose total influence is very close to the minimum possible, namely, for $pI^{p}[f]\leq\mu_{p}[f]\left(\log_{p}(\mu_{p}[f])+\epsilon\right)$, where $\epsilon\leq c_{0}/\ln(1/p)$ and $c_0$ is very small. It is likely that the conclusion holds under the weaker assumption $\epsilon < 1/\ln(1/p)$. Such an extension is not known even for the uniform measure. (See, for example, the conjectures in \cite{Ellis}.)

\end{document}